\documentclass[12pt, reqno]{amsart}

\usepackage{amsmath}
\usepackage{amsthm}
\usepackage{amssymb}
\usepackage{amsrefs}
\usepackage{mathrsfs}
\usepackage[usenames]{color}

 \newtheorem{thm}{}[section]
 \newtheorem{theorem}[thm]{Theorem}
 \newtheorem{corollary}[thm]{Corollary}
 \newtheorem{lemma}[thm]{Lemma}
 \newtheorem{proposition}[thm]{Proposition}

 \theoremstyle{remark}
 \newtheorem{remark}[thm]{Remark}
 
 \newtheorem{example}[thm]{Example}
  \newtheorem{definition}[thm]{Definition}
 
 \numberwithin{equation}{section}
 \allowdisplaybreaks
 
 \newcommand{\EE}{\ensuremath{\mathcal{E}}}
 
 \newcommand{\dd}{\ensuremath{\mathbf{d}}}
\newcommand{\DD}{\ensuremath{\mathcal{D}}}
 \newcommand{\ww}{\ensuremath{\mathbf{w}}}
  
 \newcommand{\JJ}{\ensuremath{\mathcal{J}}}
  \newcommand{\FF}{\ensuremath{\mathbb{F}}}
\newcommand{\PX}{\ensuremath{\mathcal{W}}}
\newcommand{\BV}{\ensuremath{\mathrm{BV}}}
\newcommand{\VMO}{\ensuremath{\mathrm{VMO}}}
\newcommand{\SL}{\ensuremath{\mathscr{L}}}
 \newcommand{\SSB}{\ensuremath{\mathcal{S}}}
 \newcommand{\SSS}{\ensuremath{\mathbb{S}}}
  \newcommand{\sss}{\ensuremath{\mathbf{s}}}
 \newcommand{\LL}{\ensuremath{\mathcal{L}}}

 \newcommand{\Id}{\operatorname{Id}}
 \newcommand{\NN}{\ensuremath{\mathbb{N}}}
 
\newcommand{\XX}{\ensuremath{\mathbb{X}}}
\newcommand{\YY}{\ensuremath{\mathbb{Y}}}
\newcommand{\yy}{\ensuremath{\mathbf{y}}}

\newcommand{\BB}{\ensuremath{\mathcal{B}}}
\newcommand{\OO}{\operatorname{\mathcal{O}}}
\newcommand{\ee}{\ensuremath{\mathbf{e}}}
\newcommand{\WW}{\ensuremath{\mathbb{W}}}
\newcommand{\vv}{\ensuremath{\mathbf{v}}}
\newcommand{\VV}{\ensuremath{\mathbb{V}}}
\newcommand{\VB}{\ensuremath{\mathcal{V}}}
\newcommand{\ZZ}{\ensuremath{\mathbb{Z}}}
\newcommand{\xx}{\ensuremath{\mathbf{x}}}
\newcommand{\zz}{\ensuremath{\mathbf{z}}}
\newcommand{\RR}{\ensuremath{\mathbb{R}}}

\newcommand{\supp}{\operatorname{supp}}
\newcommand{\Av}{\operatorname{Av}}

\begin{document}

\title[Banach spaces with highly conditional quasi-greedy bases]{Building  highly conditional almost greedy and quasi-greedy bases in Banach spaces}

\author[F. Albiac]{F. Albiac}\address{Mathematics Department\\ Universidad P\'ublica de Navarra\\
 Pamplona 31006\\  Spain} 
\email{fernando.albiac@unavarra.es}
 
\author[J. L.  Ansorena]{J. L. Ansorena}\address{Department of Mathematics and Computer Sciences\\
Universidad de La Rioja\\
 Logro\~no 26004\\ Spain} 
\email{joseluis.ansorena@unirioja.es}
 
\author[S. Dilworth]{S. J. Dilworth}\address{Department of Mathematics\\ University of South Carolina \\ Columbia
SC 29208 \\ USA} 
\email{dilworth@math.sc.edu}
 
 \author[ Denka Kutzarova]{Denka Kutzarova}
 \address{ 
 Department of Mathematics University of Illinois at Urbana-Champaign  \\ 
 Urbana, IL 61801 \\ USA \\and
Institute of Mathematics\\ Bulgarian Academy of Sciences\\
  Sofia\\ Bulgaria.
  } \email{denka@math.uiuc.edu}



\begin{abstract} It is known that for a conditional quasi-greedy basis $\BB$ in a Banach space $\XX$,  the associated sequence $(k_{m}[\BB])_{m=1}^{\infty}$ of its conditionality constants verifies the estimate $k_{m}[\BB]=\OO(\log m)$ and that if the reverse inequality $\log m =\OO(k_m[\BB])$ holds then $\XX$ is non-superreflexive. Indeed, it is known that a quasi-greedy basis in a superreflexive quasi-Banach space fulfils  the estimate  $k_{m}[\BB]=\OO(\log m)^{1-\epsilon}$ for some $\epsilon>0$.  However, in the existing literature one finds very few instances of spaces possessing quasi-greedy basis with  conditionality constants ``as large as possible."  Our goal in this article is to fill this gap. To that end we enhance and exploit a   technique developed  by Dilworth et al. in \cite{DKK2003} and craft  a wealth of new examples of both non-superreflexive classical Banach  spaces having quasi-greedy bases $\BB$ with $k_{m}[\BB]=\OO(\log m)$ and superreflexive classical Banach  spaces having for every $\epsilon>0$ quasi-greedy bases $\BB$ with $k_{m}[\BB]=\OO(\log m)^{1-\epsilon}$. Moreover, in most cases those bases will be almost greedy.  \end{abstract}

\thanks{F. Albiac and J.L. Ansorena are partially supported by the Spanish Research Grant \textit{An\'alisis Vectorial, Multilineal y Aplicaciones}, reference number MTM2014-53009-P. F. Albiac also acknowledges the support of Spanish Research Grant \textit{ Operators, lattices, and structure of Banach spaces},  with reference MTM2016-76808-P. S. J. Dilworth was supported by the National Science Foundation under Grant Number DMS--1361461. S. J. Dilworth and Denka Kutzarova were supported by the Workshop in Analysis and Probability at Texas A\&M University in 2017.}

\maketitle

\section{Introduction}\label{Introduction}

\noindent Let $\BB=(\xx_j)_{j=1}^\infty$ be a (Schauder) basis for a Banach space $\XX$ and let $\BB^*=(\xx_j^*)_{j=1}^\infty$ be its sequence of coordinate functionals. By default  all our bases (and basic sequences) will be assumed to be \textit{semi-normalized}, i.e., 
$0<\inf_{j\in\NN} \Vert \xx_j\Vert \le \sup_{j\in\NN} \Vert \xx_j\Vert<\infty$. 
We will denote by $S_m[\BB,\XX]$ the $m$-th partial sum projection with respect to $\BB$, i.e., 
\[
S_m[\BB,\XX](f)=\sum_{j=1}^m \xx_j^*(f) \, \xx_j, \quad f\in\XX.
\]
Given  a subset $A$ of $\NN$, the \textit{coordinate projection} on $A$ is (when well defined) the linear operator $S_A[\BB,\XX]\colon \XX\to \XX$ given by
 \[
 f\mapsto \sum_{j\in A} \xx_j^*(f) \, \xx_j.
\]
A basis $\BB=(\xx_j)_{j=1}^\infty$ is \textit{bimototone} if $\Vert S_A[\BB,\XX]\Vert\le 1$ for every  interval $A$ of integers in $\NN$. 
Every semi-normalized basis  in a Banach space $\XX$ is normalized (i.e., $\Vert x_{j}\Vert =1$ for all $j$) and bimonotone under a suitable renorming of $\XX$.

A basis $\BB$  is \textit{unconditional} if and only if
 $\sup_{A \text{ finite}} \Vert S_A[\BB,\XX]\Vert<\infty$. Thus, in some sense, the \textit{conditionality} of   $\BB$ can be measured in terms of the growth of the sequence 
\[
k_m[\BB,\XX]:=\sup_{|A|\le m} \Vert S_A[\BB,\XX]\Vert, \quad m\in\NN.
\]
Recall that a  basis $\BB=(\xx_j)_{j=1}^\infty$   is said to be {\it quasi-greedy} if it is semi-normalized  and 
 there is a constant $C$ such that 
 \begin{equation}\label{DefAG}
\Vert f - S_F[\BB,\XX](f)\Vert \le C \Vert f \Vert
\end{equation}
whenever $f\in \XX$ and  $F\subseteq \NN$ are such that $|\xx_j^*(f)|\le |\xx_k^*(f)|$ for all $j\in \NN\setminus F$ and all $k\in F$.  The least constant $C$ such that \eqref{DefAG} holds  is known as the quasi-greedy constant of the basis (see \cite{AA2017bis}*{Remark 4.2}).

The next theorem summarizes the connection that exists between  superreflexivity  and the conditionality constants of quasi-greedy  bases. 
Recall that a space $\XX$ is said to be {\it superreflexive} if every Banach space finitely representable in $\XX$ is reflexive.
\begin{theorem}[see \cites{DKK2003,GW2014,AAGHR2015,AAW2017}]\label{CharacterizationSR}  Let $\XX$ be a Banach space.
\begin{enumerate}
\item[(a)] If $\BB$ is a quasi-greedy basis for $\XX$ then $k_m[\BB,\XX]\lesssim \log m$ for $m\ge 2$. 

\item[(b)] If $\XX$ is non-superreflexive then there is a quasi-greedy basis $\BB$ for a Banach space $\YY$ finitely representable in $\XX$ with $k_m[\BB,\YY] \gtrsim \log m $  for $m\ge 2$.

\item[(c)] If $\XX$ is superreflexive and $\BB$ is quasi-greedy then there is 
$0<a<1$ such that  $k_m[\BB]\lesssim (\log m)^a$ for $m\ge 2$. 

\item[(d)] For every $0<a<1$ there is a quasi-greedy basis $\BB$ for a Banach space $\YY$ (namely, $\YY=\ell_2$) finitely representable in $\XX$  with  $k_m[\BB,\YY] \gtrsim (\log m)^a$ for $m\ge 2$.
\end{enumerate}
\end{theorem}
Theorem~\ref{CharacterizationSR}  characterizes both the superreflexivity and the lack of superreflexivity of a Banach space $\XX$ in terms of the growth of the conditionality constants of quasi-greedy bases. It could be argued that   the quasi-greedy bases whose existence is guaranteed in parts (b) and (d) lie outside  the space $\XX$, and that,  although this approach is consistent when dealing with ``super'' properties, in truth it does  not  tackle the question of the existence of a quasi-greedy basis  with large conditionality constants in the space $\XX$ itself! Hence this discussion naturally leads to the following two questions
relative to a given Banach space:

 \noindent\textbf{Question A}. Pick a  non-superreflexive Banach space $\XX$ with a quasi-greedy basis. Is there a quasi-greedy basis $\BB$ for $\XX$ with $k_m[\BB,\XX]\approx \log m$ for $m\ge 2$?

 \noindent\textbf{Question B}. Pick a  non-superreflexive Banach space $\XX$ with a quasi-greedy basis. 
  Given $0<a<1$, is there a quasi-greedy basis $\BB$ for $\XX$ with $k_m[\BB,\XX]\gtrsim (\log m)^a$ for $m\ge 2$?
 
Questions A and B can be regarded as a development of the query initiated by Konyagin and Telmyakov in 1999 \cite{KonyaginTemlyakov1999} of finding conditional quasi-greedy bases in general Banach spaces, and which has evolved  towards the more specific quest  of finding quasi-greedy bases  ``as conditional as possible.''  The reader will find a detailed account of this process in  the papers \cites{Wo2000,DKW2002,DKK2003,Gogyan2010, GHO2013, AAW2017}.

Let us outline the state of the art of those two questions. Garrig\'os and Wojtaszczyk  proved in \cite{GW2014}   that Question B has a positive answer for  $\XX=\ell_p$, $1<p<\infty$. As for Question A, it is known that   Lindenstrauss' basic sequence in $\ell_1$, the Haar system  in $\BV(\RR^d)$ for $d\ge 2$,  and  the unit-vector system in the Konyagin-Telmyakov space  $KT(\infty,p)$ for $1<p<\infty$, are all quasi-greedy basic sequences with conditionality constants as large as possible (see \cites{GHO2013,BBGHO2018}). Moreover, in \cite{GW2014} it is proved that the answer to Question A is positive for $\ell_1\oplus\ell_2\oplus c_0$, and in \cite{AAGHR2015} that  the same holds true for mixed-norm spaces  of the form $(\bigoplus_{n=1}^\infty \ell_1^n)_q$ ($1<q<\infty$), providing this way the first-known examples of reflexive Banach spaces having quasi-greedy bases with conditionality constants as large as possible. More recently, the authors constructed in \cite{AAW2017}  the first-known examples of Banach spaces  of nontrivial type and nontrivial cotype for which the answer to Question A is positive. These   spaces are  $ \PX_{p,q}^0\oplus\PX_{p,q}^0\oplus\ell_2$, $1<p,q<\infty$,   where $\PX^0_{p,q}$ and  $\PX_{p,q}$ are the interpolation spaces 
 \begin{equation}\label{PXSpaces}
  \PX^0_{p,q}=(v_1^0,c_0)_{\theta,q}, \quad\text{and}\quad \PX_{p,q}=(v_1,\ell_\infty)_{\theta,q}, \quad p=1/(1-\theta),
  \end{equation}
   defined from the space of sequences of bounded variation
\[
v_1=\{ (a_j)_{j=1}^\infty \colon |a_1|+\sum_{j=2}^\infty |a_{j}-a_{j-1}|<\infty\},
\]
and the subspace $v_1^0$ of $v_1$ resulting from  the intersection of $v_{1}$ with $c_{0}$. 
Here and throughout this paper,  $(\XX_0,\XX_1)_{\theta,q}$ denotes the Banach space obtained by applying the real interpolation method  to the Banach couple $(\XX_0,\XX_1)$ with indices $\theta$ and $q$. 
Let us recall that, in light of \cite{DKK2003}*{Theorem 8.5}, the only quasi-greedy basis for a Banach space  whose dual has the Grothendieck Theorem Property is the unit vector basis of $c_0$. Consequently, the answer to Question A is negative for 
$c_0$.  Any other  $\LL_\infty$-space,  despite having a basis (see \cite{JRZ1971}*{Theorem 5.1}), does not have a quasi-greedy basis.
 
 In this article we develop the necessary  machinery that permits to extend  the scant list of known Banach spaces for which the answer either to Question A or to Question B is positive. Moreover, in a wide class of Banach spaces, 
 the examples of bases we provide are not only quasi-greedy but are  \textit{almost greedy}. Recall that a basis $\BB=(\xx_j)_{j=1}^\infty$ for a Banach space $\XX$ is \textit{almost greedy} if there is a constant $C$ such that
\[
\Vert f - S_F[\BB,\XX](f)\Vert \le C \Vert f - S_A[\BB,\XX](f)\Vert
\]
whenever $f\in\XX$, $|A|\le |F|<\infty$, and $|\xx_j^*(f)| \le |\xx_k^*(f)|$ for any $j\in \NN\setminus F$ and $k\in F$. Almost greedy bases were characterized in  \cite{DKKT2003} as those bases that are simultaneously quasi-greedy and democratic. In fact,  in this characterization democracy can be replaced with super-democracy. 
A basis $\BB$ is \textit{super-democratic} if there is a sequence  $(\lambda_m)_{m=1}^\infty$ such that
\[\left\Vert \sum_{j\in A} \varepsilon_j \xx_j\right\Vert \approx \lambda_{|A|}\]
 for any $A\subseteq\NN$ finite and any $(\varepsilon_j)_{j\in A}$ sequence of signs, in which case $(\lambda_m)_{m=1}^\infty$ is equivalent to its fundamental function, defined by
\[
\varphi_m[\BB,\XX]=\sup_{|A|\le m}\left\Vert \sum_{j\in A} \xx_j\right\Vert, \quad m\in\NN.
\]

 A more demanding concept than super-democracy is that of {bi-democracy}. A basis $\BB$ is said to be \textit{bi-democratic} if
\[
\varphi_m[\BB,\XX] \, \varphi_m[\BB^*,\YY]\lesssim m \text{ for } m\in \NN,
\]
where  $\YY$ is the closed linear span of the basic sequence $\BB^*$ in $\XX^*$.

 Our study includes, among other spaces, the finite direct sums  
 \[
 D_{p,q}:=\begin{cases} 
 \ell_p\oplus\ell_q  \text{ if }1\le p,q<\infty,  \\  
 \ell_p\oplus c_0 \text{  if } 1\le p<\infty \text{  and  } q=0,
  \end{cases}
 \]
 the matrix spaces 
 \[
 Z_{p,q}:=\begin{cases}
  \ell_q(\ell_p) \text{  if  }1\le p,q<\infty,  \\  
  c_0(\ell_p)  \text{  if  } 1\le p<\infty \text{  and  } q=0, \\
 \ell_q(c_0) \text{  if  } p=0 \text{  and  } 1\le q<\infty, 
 \end{cases}
 \]
 and the mixed-norm spaces of the family
 \[
B_{p,q}:=(\oplus_{n=1}^\infty \ell_p^{n})_q, \,  p\in[1,\infty],  \, q\in\{0\}\cup[1,\infty), \,q\not=0 \text{ when }p=\infty.
\]
We use  $(\oplus_{n=1}^\infty \XX_n)_q$ to denote the direct sum of the Banach spaces $\XX_n$ in the $\ell_q$-sense ($c_0$-sense if $q=0$).

In this article we shall  prove that the answer  to Question A is positive for  all non-superreflexive spaces in the aforementioned list. We also show that in all superreflexive spaces in the above list the answer to Question B is positive. These  results will appear in Section~\ref{Main}. Previously, in Sections~\ref{Preliminaries} and \ref{DKKMethod}, we introduce the tools that we will use to achieve our goal.

Throughout this article we follow standard Banach space terminology and notation as can be found, e.g.,  in \cite{AlbiacKalton2016}. In what follows we would like to single out the notation and terminology that  is more commonly employed. We deal with real or complex Banach spaces, and $\FF$ will  denote the underlying scalar field. A weight will be a sequence of positive scalars.
Given families of positive real numbers $(\alpha_i)_{i\in I}$ and $(\beta_i)_{i\in I}$, the symbol $\alpha_i\lesssim \beta_i$ for $i\in I$ means that $\sup_{i\in I}\alpha_i/\beta_i <\infty$. If $\alpha_i\lesssim \beta_i$ and $\beta_i\lesssim \alpha_i$ for $i\in I$ we say $(\alpha_i)_{i\in I}$ are $(\beta_i)_{i\in I}$
 are equivalent, and we write $\alpha_i\approx \beta_i$ for $i\in I$. Applied to Banach spaces, the symbol $\XX\approx \YY$ means that the spaces $\XX$ and $\YY$ are isomorphic, while the symbol $\XX\lesssim_c\YY$ means that  $\XX$ is isomorphic to a complemented subspace of $\YY$. 
The norm of a linear operator $T\colon \XX\to\YY$ will be denoted either by $\Vert T\Vert_{\XX\to\YY}$ or, when the Banach spaces $\XX$ and $\YY$ are clear from context, simply by $\Vert T\Vert$.
Given families of Banach spaces
$(\XX_i)_{i\in I}$ and $(\YY_i)_{i\in I}$, the symbol $\XX_i\lesssim_c \YY_i$ for $i\in I$ means that  the spaces
$\XX_i$ are uniformly isomorphic to  complemented subspaces of $\YY_i$, i.e., there are  linear operators
$L_i\colon \XX_i \to \YY_i$, $T_i\colon \YY_i \to \XX_i$ such that $T_i\circ L_i=\Id_{\XX_i}$ and $\sup_i \Vert T_i\Vert \Vert L_i\Vert<\infty$. Similarly the symbol $\XX_i\approx \YY_i$ for $i\in I$ means that Banach-Mazur distance from $\XX_i$ to $\YY_i$ is uniformly bounded. We write $\XX\oplus\YY$ for the Cartesian product of the Banach spaces $\XX$ and $\YY$ endowed with the norm 
\[
\Vert (x,y)\Vert_{\XX\oplus\YY}= \Vert x\Vert + \Vert y\Vert, \quad x\in\XX,\ y\in\YY. 
 \]
As it is customary,  we put $\delta_{j,k}=1$ if $j=k$ and $\delta_{j,k}=0$ otherwise. Given $j\in\NN$, the $j$-th unit vector is defined by $\ee_j=(\delta_{j,k})_{k=1}^\infty$ and  the \textit{unit-vector system} will be sequence $\EE:=(\ee_j)_{j=1}^\infty$.
We denote by $\ee_j^*$ the $j$-th coordinate functional  defined on $\FF^\NN$ by $(a_k)_{k=1}^\infty\mapsto  a_j$, and by $S_A\colon \FF^\NN\to \FF^\NN$ the coordinate projection on a set $A\subseteq\NN$. Given $m\in\NN$, $S_m$ will be be coordinate projection on the set $\{1,\dots,m\}$.

The linear span of a family $(x_i)_{i\in I}$ in a Banach space will be denoted by $\langle x_i \colon i\in I\rangle$, and its closed linear span by $[x_i \colon i\in I]$.  Other more specific notation  will be introduced on the spot when needed.

\section{Preliminary results}\label{Preliminaries}

\noindent Most of the ideas behind the results we include in this preliminary section have appeared more or less explicitly  in the literature before. Nonetheless, for the sake of clarity and completeness, we shall include the statements of the results we need in the form that best suits our purposes and  the sketches of their proofs. 

\begin{definition} Given a basis $\BB$ for a Banach space $\XX$,
we define the  sequence $(L_m[\BB,\XX])_{m=1}^{\infty}$    by
\begin{equation*}
L_m[\BB,\XX]=\sup\left\{ \frac{ \Vert S_A[\BB,\XX](f)\Vert}{ \Vert f \Vert } \colon \max(\supp(f))\le m, \, A\subseteq\NN\right\}.
\end{equation*}
\end{definition}
When  the space $\XX$ is clear from context we will drop it  and simply write $L_m[\BB]$. 
Likewise, we will  drop  $\XX$ from the notation of
all the other  concepts  that we introduced in Section~\ref{Introduction} involving a basis $\BB$ for a Banach space.

Notice that $\BB$ is unconditional if and only if $\sup_m L_m[\BB]<\infty$. Hence, the growth of the sequence $(L_m[\BB])_{m=1}^{\infty}$ provides also a measure of the conditionality of the basis. Since $L_m[\BB]\le k_m[\BB]$ for all $m\in\NN$, any  result establishing that the size of the members of the sequence $(L_m[\BB])_{m=1}^\infty$ is large, is   a stronger statement than the corresponding one enunciated in terms of 
$(k_m[\BB])_{m=1}^\infty$. 

The papers \cites{AAW2017,GW2014} draw attention to the fact that, in some cases, the sequence $(L_m[\BB])_{m=1}^{\infty}$ is more fit than the ``usual'' sequence of conditionality constants $(k_m[\BB])_{m=1}^\infty$ for transferring conditionality properties from a given basis to a basis constructed from it. This is the reason why we establish all the instrumental results of this section  in terms of  the conditionality constants $(L_m[\BB])_{m=1}^{\infty}$. Notice that, in contrast  to $(k_m[\BB])_{m=1}^\infty$,  the sequence $(L_m[\BB])_{m=1}^\infty$ is not necessarily doubling. This   leads us to use a doubling function in our statements. 
Recall that a function $\delta\colon[0,\infty)\to[0,\infty)$  is said to be \textit{doubling} if for some non-negative constant $C$ one has $\delta(2t)\le C \delta(t)$ for all $t\ge 0$. 

Given sequences $\BB_0=(\xx_j)_{j=1}^\infty$ and $\BB_1=(\yy_j)_{j=1}^\infty$ in Banach spaces $\XX$ and $\YY$, respectively,  their \textit{direct sum} $\BB_0 \oplus \BB_1$ will be the sequence in $\XX\times\YY$ given by
\begin{equation*}
\BB_0 \oplus \BB_1=((\xx_1,0),(0,\yy_1),(\xx_2,0),(0,\yy_2),\dots,((\xx_j,0),(0,\yy_j),\dots).
\end{equation*}
\begin{lemma}
[cf. \cite{GHO2013}]
\label{LemmaOne} Suppose that $\BB_0=(\xx_j)_{j=1}^\infty$ and $\BB_1=(\yy_j)_{j=1}^\infty$ are bases for some Banach spaces $\XX$ and $\YY$, respectively. Assume that 
 $L_m[\BB_0]\gtrsim \delta(m)$ for $m\in\NN$ for some non-decreasing   doubling function  $\delta\colon[0,\infty)\to[0,\infty)$. Then $\BB_0 \oplus \BB_1$ is a basis for $\XX\oplus\YY$ with 
$L_m[\BB_0\oplus\BB_1]\gtrsim \delta(m)$ for $m\in\NN$. Moreover:
 \begin{enumerate} 
 \item[(a)] If $\BB_0$ and $\BB_1$ are quasi-greedy so is $\BB_0\oplus\BB_1$.
 \item[(b)] If $\BB_0$ and $\BB_1$ are  super-democratic, both with fundamental function equivalent to  $(\lambda_m)_{m=1}^\infty$, then $\BB_0\oplus\BB_1$ is super-democratic with fundamental function equivalent to $(\lambda_m)_{m=1}^\infty$\end{enumerate}\end{lemma}

\begin{proof} It is similar to the proof of \cite{GHO2013}*{Proposition 6.1}, so we omit it. \end{proof}

Our next lemma follows an idea from \cite{Wo2000} for constructing quasi-greedy bases. To state it properly, it will be convenient to introduce some additional notation. 
Given $N\in\NN$ and a Banach space $\XX\subseteq\FF^\NN$  for which 
the unit-vector system is a basis, $\XX^{(N)}$ will be the $N$-dimensional space $[\ee_j \colon 1\le j \le N]$ regarded as a subspace of $\XX$. As it is customary, we will write 
\[\ell_p^N:=\ell_p^{(N)},\;1\le p<\infty,\quad \text{and}\quad \ell_\infty^N:=c_0^{(N)}.\]
More generally, given a basis $\BB=(\xx_j)_{j=1}^\infty$ for a Banach space $\XX$ and  $N\in\NN$ we will consider the closed linear span of the truncated finite sequence $
(\xx_{j})_{j=1}^{N}$, i.e.,
\[
\XX^{(N)}[\BB]=[\xx_j \colon 1\le j \le N].
\]

Let $\BB= (\xx_j)_{j=1}^\infty$ be a basis in a Banach space $\XX$. Given a sequence of positive integers  $(N_n)_{n=1}^\infty$ we define a sequence $(\zz_k)_{k=1}^\infty$ that we denote $\bigoplus_{n=1}^\infty
(\xx_{j})_{j=1}^{N_n}
$ in the space $\XX^\NN$    by 
\begin{equation*}
\zz_k=(\underbrace{0,\dots,0}_{r-1 \text{ times}},\xx_j,0,\dots,0,\dots),
\end{equation*}
where, for a given $k\in\NN$, the integers $r$ and $j$ are univocally determined by the relations $k=j+\sum_{n=1}^{r-1} N_n$ and $1\le j\le N_n$.

\begin{lemma}\label{LemmaTwo} Suppose that $\BB=(\xx_j)_{j=1}^\infty$ is a basis for a Banach space $\XX$ with $L_m[\BB]\gtrsim \delta(m)$ for $m\in\NN$, for some  non-decreasing   doubling function  $\delta\colon[0,\infty)\to[0,\infty)$.  
Let $(N_n)_{n=1}^\infty$ be a sequence of positive integers verifying
 \[
 M_r:=\sum_{n=1}^{r}  N_n \lesssim N_{r+1} \text{ for } r\in\NN.
 \]
Then
\begin{enumerate} 
\item[(a)] The sequence $\BB_0=\bigoplus_{n=1}^\infty 
(\xx_{j})_{j=1}^{N_n}
$ is a  basis for  the Banach space $(\bigoplus_{n=1}^\infty \XX^{(N_n)}[\BB])_p$,  $p\in \{0\}\cup[1,\infty)$,  with $L_m[\BB_0] \gtrsim \delta(m)$ for $m\in\NN$.
\item[(b)] If $\BB$ is quasi-greedy so is $\BB_0$.
\item[(c)]  If $\BB$ is super-democratic with fundamental function equivalent to  $(m^{1/p})_{m=1}^\infty$ for some $1\le p<\infty$,
so is  $\BB_0$.
\end{enumerate}
 \end{lemma}
 
\begin{proof}
It is clear that $\BB_0$ is a quasi-greedy basis with the same quasi-greedy constant as $\BB$, hence we need only take care of obtaining an estimate for  $(L_m[\BB_0])_{m=1}^{\infty}$. 
Let $C_1$ be such that $C_1 L_m[\BB]\ge  \delta(m)$ for all $m\in\NN$. Let $C_2>1$ be such that
$M_r\le C_2 N_r$ for all $r\in\NN$. Since $\delta$ is doubling there is a constant $C_3$ such that  $C_3 \delta(m)\ge \delta (2 C_2 m)$ for all $m\in\NN$. 
Given $m\ge M_1$, pick $r\in\NN$ such that
$M_{r}\le  m < M_{r+1}$. We have
\begin{align*}
C_1 C_3 L_m[\BB_0]
&\ge C_1 C_3 \max\{L_{N_r}[\BB], L_{m-M_r}[\BB]\}\\
&\ge C_3\max\{\delta(N_r), \delta(m-M_r)\}\\
&\ge \max\{\delta(2 C_2 N_r), \delta(2C_2(m-M_r))\}\\
&\ge \max\{\delta(2 M_r), \delta(2m-2M_r))\}\\
&=\delta(  \max\{ 2 M_r, 2m-2M_r\})\\
&\ge \delta(m),
\end{align*}
as desired. 
\end{proof}

The spaces $\PX^0_{p,q}$  and $\PX_{p,q}$ ($1<p<\infty$, $1\le q <\infty$) defined  in \eqref{PXSpaces} were introduced and studied by Pisier and Xu \cite{PisierXu1987}.  It is verified that  
$\PX^0_{p,q} \approx \PX_{p,q}$. Moreover, when $q>1$  these spaces have nontrivial type and nontrivial cotype  and  they are pseudo-reflexive. 

Our next proposition is a new addition  to the study of Pisier-Xu spaces, which will be used below. Recall that given $1\le q<\infty$ and a scalar sequence $\ww=(w_n)_{n=1}^\infty$, the  Lorentz sequence space $d_q(\ww)$ consists of all sequences $f$ in $c_{0}$ whose non-increasing rearrangement $(a_n^*)_{n=1}^\infty$ verifies
 \[
\Vert f \Vert_{d_q(\ww)}= \left( \sum_{n=1}^\infty (a_n^*)^q w_n \right)^{1/q} <\infty.
\]
In  the case when $\ww=(n ^{q/p-1})_{n=1}^\infty$  for some $1\le p<\infty$ we have that $\ell_{p,q}:=d_q(\ww)$ is the classical sequence Lorentz space of indices $p$ and $q$.

\begin{proposition}\label{PXLorentz}Let $1<p<\infty$ and $1\le q <\infty$.
Then 
\[\ell_{p,q}\lesssim_c \PX_{p,q}^0.\] In fact, $(\ee_{2j-1})_{j=1}^\infty$ is a complemented basic sequence
isometrically equivalent to the unit vector system in $\ell_{p,q}$.
\end{proposition}
\begin{proof} Put $p=1/(1-\theta)$. Consider the linear maps $L,T\colon \FF^\NN\to \FF^\NN$ defined by
\begin{align}
L((a_j)_{j=1}^\infty)&=(a_1,0,a_2,0,\dots,a_j,0,a_{j+1},0,\dots), \label{Lifting}\\
T((a_j)_{j=1}^\infty)&=(a_1-a_2,a_3-a_4,\dots, a_{2j-1}-a_{2j}, \dots).\label{Retraction}
\end{align}
We have $\Vert L \colon  \ell_1 \to v_1^0\Vert\le 1, \Vert L \colon c_0 \to c_0\Vert\le 1, \Vert T \colon v_1^0 \to \ell_1\Vert \le 1$, and
$\Vert T \colon  c_0 \to c_0\Vert \le 2$. Taking into account that \[
(\ell_1,c_0)_{\theta,q}=(\ell_1,\ell_\infty)_{\theta,q}=\ell_{p,q}\] (see, e.g., \cite{BenSha1988}*{Theorem 1.9}),   interpolation gives
$\Vert L \colon \ell_{p,q} \to \PX^0_{p,q} \Vert\le 1$ and $\Vert T \colon \PX^0_{p,q}  \to \ell_{p,q} \Vert\le 2^{\theta}$.
Since $T (L(f))=f$ for every $f\in\FF^\NN$ we are done.
\end{proof}

\begin{corollary} Let $1<p<\infty$ and $1\le q <\infty$. Then
$\ell_q \lesssim_c \PX^0_{p,q}$.
\end{corollary}
\begin{proof}In light of Proposition~\ref{PXLorentz}, it suffices to see that $\ell_q\lesssim_c \ell_{p,q}$.
 By \cite{LinTza1972}*{Proposition 4} we have $\ell_q\lesssim_c \ell_{p,q}$
if $q\le p$ and  $\ell_{q'} \lesssim_c \ell_{p',q'}$ otherwise. We conclude the proof by dualizing  (see \cite{Allen1978}*{Theorem 1}).
\end{proof}

\section{
The Dilworth-Kalton-Kutzarova method, revisited
}\label{DKKMethod}
\noindent
Recall that a basis is said to be \textit{subsymmetric} if it is unconditional and equivalent to all of its subsequences.
 If $(\xx_j)_{j=1}^\infty$ is a subsymmetric  basis in a Banach space  $(\SSS,\Vert \cdot\Vert_\SSS)$ then there is a constant $C$ such that
\[
C^{-1}\left\Vert\sum_{j=1}^\infty a_j \xx_j \right\Vert_\SSS
\le \left\Vert\sum_{j=1}^\infty \varepsilon_j a_j \xx_{\phi(j)} \right\Vert_\SSS
\le C\left\Vert\sum_{j=1}^\infty a_j \xx_j \right\Vert_\SSS
\]
for all $\sum_{j=1}^\infty a_j \xx_j\in\SSS$, all  sequence of signs $(\varepsilon_j)_{j=1}^\infty$, and all increasing 
maps $\phi$,
in which case the basis is said to be $C$-subsymmetric.
Every subsymmetric basis is quasi-greedy and super-democratic, hence almost greedy. 

A \textit{subsymmetric} 
\textit{sequence space} will be a Banach space $\SSS\subseteq\FF^\NN$ for which the unit-vector system is a $1$-subsymmetric 
basis. Note that every Banach space equipped with a subsymmetric 
basis is isomorphic to a subsymmetric 
sequence space   (see, e.g., 
\cite{AnsoQM}).

Given $f\in\FF^\NN$ and $A\subseteq\NN$ finite, we put
\[
\Av(f,A)= \frac{1}{|A|} \left(\sum_{j\in A}  a_j\right).
\]
The 
\textit{averaging projection} with respect to a sequence $\sigma=(\sigma_n)_{n=1}^\infty$  of disjoint finite subsets of $\NN$ is the map $P_\sigma\colon \FF^\NN\to\FF^\NN$ defined by
\[
(a_j)_{j=1}^\infty\mapsto (b_k)_{k=1}^\infty, \quad b_k =\Av(f,\sigma_n)
\text{ if }k\in \sigma_n.
\]
We denote by $Q_{\sigma}=\Id_{\FF^\NN}-P_{\sigma}$ its complementary projection.
A sequence  $(\sigma_n)_{n=1}^\infty$  of subsets of $\NN$
is said to be  \textit{ordered} when\[
\max(\sigma_{n})<\min(\sigma_{n+1}), \quad n\in\NN.
\] 
Note that any ordered sequence consists of disjoint finite subsets.
Let us recall the following result (see \cite{LinTza1977}*{Propostion 3.a.4}).
\begin{theorem}\label{AveragingSubsymmetric}
 Let $(\SSS,\Vert \cdot\Vert_\SSS)$ be a subsymmetric  sequence space and $\sigma$  an ordered  sequence of subsets of $\NN$. Then $P_\sigma$ is bounded from $\SSS$ into $\SSS$ with $\Vert P_\sigma\Vert_{\SSS\to\SSS}\le 2$.
 \end{theorem}
An \textit{ordered partition} of $\NN$ will be an ordered sequence $(\sigma_n)_{n=1}^\infty$  of subsets of $\NN$ with
$\NN=\cup_n \sigma_n$. Notice that ordered partitions consist of integer intervals.
In fact, if  $\sigma=(\sigma_n)_{n=1}^\infty$ and we let
\[
M_r=\sum_{n=1}^r |\sigma_n|,\quad r\in\NN\cup\{0\},
\]
we have $\sigma_n=[1+M_{n-1},M_n]$.
 Theorem~\ref{AveragingSubsymmetric} immediately gives
\begin{equation}
\label{ComplementAveragingSubsym}
\Vert P_\sigma (f)\Vert_\SSS \le 2 \Vert f \Vert_\SSS,\quad
 \Vert Q_\sigma (f)\Vert_\SSS \le 3 \Vert f \Vert_\SSS, \quad f\in c_{00}
\end{equation}
for any ordered partition $\sigma$ and any subsymmetric sequence space $\SSS$.
Another consequence of Theorem~\ref{AveragingSubsymmetric} is that
subsymmetric bases are bi-democratic.
To be precise for any subsymmetric sequence space $(\SSS,\Vert \cdot\Vert_\SSS)$ we have  (see \cite{LinTza1977}*{Proposition 3.a.6})
\begin{equation}
\label{bidemocracy}
 m\le 
\left\Vert \sum_{j\in 1}^m  \ee_j \right\Vert_{\SSS}
\left\Vert \sum_{j=1}^m   \ee_j^* \right\Vert_{\SSS^*}
\le
 2 m , \quad m\in\NN.
\end{equation}
 
 \begin{remark}\label{SubSymSquare}Notice that if  $\BB$ is a subsymmetric basis for a Banach space $\SSS$ then $\SSS\oplus\SSS\approx\SSS$. To see this it suffices to consider the subsequences $\BB_o$ and $\BB_e$  
 consisting, respectively, of the odd and the even terms of $\BB$. Then,  on the one hand, we have that $\BB_o\oplus\BB_e$ is equivalent to $\BB$ and, on the other hand, it is equivalent to $\BB\oplus\BB$. 
 \end{remark}
 
Given a subsymmetric sequence space $(\SSS,\Vert \cdot\Vert_\SSS)$ and an ordered partition  $\sigma=(\sigma_n)_{n=1}^\infty$
of $\NN$ we will put
\begin{equation}\label{DemocracySubSym}
\Lambda_{m}=\left\Vert \sum_{j=1}^m \ee_j\right\Vert_\SSS, \quad 
\Lambda^*_{m}=\frac{m}{\Lambda_m}, \quad m\in\NN \text{, and }
\end{equation}
\begin{equation}\label{BOSforAP}
\vv_n=\frac{1}{\Lambda_{|\sigma_n|}} \sum_{j\in\sigma_n} \ee_j,\quad
\vv_n^*=  \frac{1}{\Lambda^*_{|\sigma_n|}}  \sum_{j\in\sigma_n} \ee^*_j, \quad n\in\NN.
\end{equation}
We have that $(\vv_n)_{n=1}^\infty$ is a normalized basic sequence in $\SSS$ and 
\[
\
P_{\sigma}(f)=\sum_{n=1}^\infty  \vv_n^*(f) \, \vv_n, \quad f\in\FF^\NN.
\]
Consequently, by  
inequality \eqref{ComplementAveragingSubsym},
\begin{equation}\label{equivalentnorm}
\
\Vert f\Vert_{\SSS} \le  
\Vert Q_{\sigma}(f)\Vert_{\SSS} + \left\Vert \sum_{n=1}^\infty \vv_n^*(f) \, \vv_n\right\Vert_\SSS
 \le 5 \Vert f\Vert_\SSS, \quad f\in\SSS.
\end{equation}
That is, the middle term in \eqref{equivalentnorm} defines an equivalent norm for $\SSS$. Replacing  the  basic sequence $(\vv_n)_{n=1}^\infty$ with an arbitrary  basis provides a method,  invented in \cite{DKK2003}, for constructing     Banach spaces with special types of bases.
Let us give a precise description of this  method.  Suppose  $\BB=(\xx_n)_{n=1}^\infty$ is a semi-normalized basis for a Banach space $(\XX,\Vert \cdot\Vert_\XX)$.
 Consider the linear mapping $H\colon c_{00} \to c_{00} \times \XX$ given by
 \begin{equation}\label{eq:1}
 f\mapsto   \left( Q_{\sigma}(f), \sum_{n=1}^\infty  \vv_n^*(f)\, \xx_n \right),
\end{equation}
and define a gauge on $c_{00}$,  by 
\[
\left\Vert  f \right\Vert_{\BB,\SSS,\sigma}
=\Vert H(f)\Vert_{\SSS\oplus\XX}
=\Vert  Q_\sigma (f) \Vert_\SSS 
+
\left\Vert \sum_{n=1}^\infty \vv_n^*(f)\, \xx_n\right\Vert_\XX.
\]

The authors of \cite{DKK2003} consider the sequence Banach space obtained by the completion of $(c_{00},\left\Vert  \cdot \right\Vert_{\BB,\SSS,\sigma})$. In order to justify the validity of this procedure,  called throughout the \textit{DKK-method}  for short, we need a couple of lemmas.
\begin{lemma}\label{lem: normequivalence5}Let $\BB$ be a (semi-normalized) basis of a Banach space $\XX$. Let  $\SSS$ be a subsymmetric sequence space  and $\sigma$ be an ordered partition of $\NN$. Then there are constants $C_1$ and $C_2$ such that
\[
C_1^{-1} \Vert f \Vert_{\SSS} \le
\Vert f\Vert_{\BB,\SSS,\sigma}
\le C_2 \Vert f \Vert_{\SSS},
\]
for all sequences $f$ supported on $\sigma_k$ for some $k\in\NN$. Moreover, if $\BB$ is normalized, then $C_1=1$ and $C_2=5$.
\end{lemma}
\begin{proof}  Assume, without loss of generality, that $\BB=(\xx_n)_{n=1}^\infty$ is normalized. Let $\vv_n$ and $\vv_n^*$ be  for $n\in\NN$  as is \eqref{BOSforAP}. Let $k\in\NN$ and $f\in c_{00}$ such that $\supp f \subseteq \sigma_k$.
Since  $\vv_n^*(f)=0$ for $n\not=k$ we have
\[
 \Vert f\Vert_{\BB,\SSS,\sigma}
 = \Vert Q_\sigma (f)\Vert_\SSS +  | \vv_k^*(f) |  \, \Vert \xx_k\Vert_\XX
 = \Vert Q_\sigma (f)\Vert_\SSS +  \left \Vert \sum_{n=1}^\infty \vv_n^*(f)  \, \vv_n\right\Vert_\SSS.
\]
 We conclude the proof by appealing to  \eqref{equivalentnorm}.\end{proof}

\begin{lemma}\label{CommutingLemma}
 Let $\SSS$ be a subsymmetric sequence space, and $\sigma=(\sigma_n)_{n=1}^\infty$ be an ordered partition of $\NN$. 
Let $A\subseteq\NN$ finite.
With $(\vv_n^*)$  as is \eqref{BOSforAP} and $B=\cup_{n\in A}\sigma_n$  we have
\[
\vv_n^*(S_{B}(f))=\begin{cases}\vv_n^*(f) & \text{if }n\in A \\ 0 & \text{otherwise,}\end{cases} 
\]
and 
\[
S_{B}(Q_\sigma(f))=Q_\sigma(S_{B}(f))\]
for all  $f\in c_{00}$.
\end{lemma}
\begin{proof} The first equality is clear from the definition and implies that $S_{B}(P_\sigma(f))=P_\sigma(S_{B}(f))$.
Thus the second identity holds.
\end{proof}

From   Lemma~\ref{lem: normequivalence5} and Lemma~\ref{CommutingLemma} we deduce the existence of a constant $C$ such that 
\[
|a_k|\le  C \Vert f\Vert_{\BB,\SSS,\sigma}
\] for all $k\in\NN$ and all $f=(a_j)_{j=1}^\infty\in c_{00}$.
We infer that $(c_{00},\Vert \cdot\Vert_{\BB,\SSS,\sigma})$ is a normed space whose completion can be carried out inside $c_0$. So we can safely give the following definition.

\begin{definition}Let $\BB$ be a basis of a Banach space $\XX$.  Let $\SSS$ be a subsymmetric sequence space and $\sigma$ be an ordered partition of $\NN$. 
\begin{itemize}
\item We define the space $\YY[\BB,\SSS,\sigma]$ as the completion of $(c_{00},\Vert \cdot\Vert_{\BB,\SSS,\sigma})$ carried out inside $c_0$.

 \item $\YY^{(N)}[\BB,\SSS,\sigma]$ will be the $N$-dimensional space $[\ee_j \colon 1\le j \le N]$ regarded as a subspace of 
$\YY[\BB,\SSS,\sigma]$.
 \end{itemize}
\end{definition}

The next theorem summarizes some early properties of the DKK-method. For expositional ease we include a sketch of their proof.

\begin{theorem}[cf. \cite{DKK2003}]\label{FirstProperties}
Let $\BB=(\xx_j)_{j=1}^{\infty}$ be a  basis for a Banach space $\XX$, let $\SSS$ be a subsymmetric sequence space, and $\sigma=(\sigma_n)_{n=1}^\infty$ be an ordered partition of $\NN$. We have:
\begin{itemize}
\item[(a)] The unit-vector system $\mathcal E=(\ee_n)_{n=1}^\infty$ is a semi-normalized basis for $\YY[\BB,\SSS,\sigma]$.
\item[(b)] $\YY[\BB,\SSS,\sigma] \approx  Q_\sigma(\SSS)  \oplus \XX $. In fact,
the mapping $H$  in \eqref{eq:1} extends to an isomorphism from 
$\YY[\BB,\SSS,\sigma]$ onto $Q_\sigma(\SSS)  \oplus \XX$.

\item[(c)] If $M_r=\sum_{n=1}^r |\sigma_n|$ for $k\in \NN$ then
\[
\YY^{(M_r)}[\BB,\SSS,\sigma]\approx Q_\sigma(\SSS^{(M_r)})\oplus \XX^{(r)}[\BB].
\]

\item[(d)] If $\BB'$ is a basis for a Banach space $\XX^{\prime}$ then  (up to an equivalent norm) $\YY[\BB,\SSS,\sigma]) = \YY[\BB',\SSS,\sigma]$ if and only if the bases $\BB$ and 
$\BB'$ are equivalent.
\end{itemize}
\end{theorem}

\begin{proof}Denote $Q=Q_\sigma$ and let $H$ be as in \eqref{eq:1}. 
By Lemma~\ref{CommutingLemma} we have
\[
H(S_{M_r}(f))=(S_{M_r}(Q(f)), S_r[\BB](f)), \quad r\in\NN, \, f\in c_{00}.
\]
Hence, the partial-sum projections $(S_{M_r})_{r=1}^\infty$ are uniformly bounded.
Combining  with Lemma~\ref{lem: normequivalence5} gives (a).

 The mapping $H$  is continuous by definition and so it 
 extends to a continuous map $\tilde H$ from $\YY[\BB,\SSS,\sigma]$ into $Q(\SSS)\oplus \XX$. Let us prove that $\tilde H$ is an onto isomorphism. Put $V=c_{00}\cap Q(\SSS)$, $X=\langle \xx_k\colon k\in\NN\rangle$, and let $L$ be the linear map from $X$ into $c_{00}$ defined by $L(\xx_n)= \vv_n$ for all $n\in\NN$. 
 It is straightforward to check that the map
 \[
 G\colon V\times X  \to c_{00}, \quad
 (g,h)\mapsto g+L(h)
 \]
is continuous and verifies $G(H(f))=f$ for all $f\in c_{00}$, and also $H(G(g,h))=(g,h)$ for all $g\in V$ and $h\in X$. Hence, in order  to obtain (b) it only remains to see that  $V$ is a dense subspace of $Q(\SSS)$. Let $f\in Q(\SSS)$ and $\epsilon>0$. Pick $g\in U$ such 
$ \Vert f-g\Vert_\SSS\le \varepsilon/\Vert Q\Vert$. Since 
\[ \Vert f -Q(g)\Vert_\SSS=\Vert Q(f) -Q(g)\Vert_\SSS\le \Vert Q\Vert \, \Vert f-g\Vert_{\SSS}\le\epsilon \]
and $Q(g)\in V$ we are done.

We obtain (c) by restricting the isomorphism $\tilde H$  to $\YY^{(M_r)}[\BB,\SSS,\sigma]$.

Assume that $\BB'=(\xx_n')_{n=1}^\infty$ is a basis for $\XX'$. Let $L'$ and $G'$ be the operators corresponding, respectively, to $L$ and $G$ when replacing $\BB$ with 
$\BB'$. We have that $\YY[\BB,\SSS,\sigma]=\YY[\BB',\SSS,\sigma]$ if and only if 
$H\circ G'$ extends to an isomorphism from $Q(\SSS)\oplus \XX$ onto $Q(\SSS)\oplus \XX'$. It is straightforward to check that
\[
\
H\left(G'\left(g,\sum_{n=1}^\infty a_n \, \xx_n\right)\right)=\left(g,\sum_{n=1}^\infty   a_n \, \xx'_n\right), \quad g\in V,\, (a_n)_{n=1}^\infty\in c_{00}.
\]
Hence, $\YY[\BB,\SSS,\sigma]=\YY[\BB',\SSS,\sigma]$ if and only if the
mapping
 \[
\sum_{n=1}^\infty a_n\, \xx_n\mapsto \sum_{n=1}^\infty a_n, \xx_n', \quad (a_n)_{n=1}^\infty\in c_{00}
\] 
extends to an isomorphism from $\XX$ onto $\XX'$.
\end{proof}

Part~(d) of Theorem~\ref{FirstProperties} alerts us that in the case when the basis $\BB$ is conditional so is the unit-vector system
of $\YY[\BB,\SSS,\sigma]$. Indeed, it is possible to obtain a relation between the conditionality constants of  both bases.
Prior to formulate this result we introduce some notation. 

\begin{lemma}\label{ConditionalityConstantsRelation}
Let $\BB$ be a  basis for a Banach space $\XX$,  let $\SSS$ be a subsymmetric sequence space, and $\sigma=(\sigma_n)_{n=1}^\infty$ be an ordered partition of $\NN$. If $M_r=\sum_{n=1}^r |\sigma_n|$ for all $r\in \NN$, then
\[
L_{M_r}[\EE,\YY[\BB,\SSS,\sigma]]\ge L_r[\BB, \XX].
\]
\end{lemma}
\begin{proof}Let $\Lambda_n$, $\vv_n$, and $\vv_n^*$ be defined for all $n\in\NN$ as in
\eqref{DemocracySubSym} and \eqref{BOSforAP}. Put $m_0=0$.
Given a non-null $r$-tuple $(a_n)_{n=1}^r$  we define a sequence $f=(b_j)_{j=1}^{\infty}$ by 
\[
b_j=\begin{cases}
a_n/\Lambda_{|\sigma_n|} & \text{if  $j\in\sigma_n$ for some $1\le n \le r$,}\\
0 & \text{if $j>M_r$.}
\end{cases}
\]
We have $\vv_n^*(f)=a_n$ for all $n\in\NN$, 
$P_\sigma (f)=f$ and, then, $Q_\sigma (f)=0$. Consequently, for any $A\subseteq\NN$, putting $B=\cup_{n\in A} \sigma_n$, we have
\[
L_{M_r}[\EE,\YY[\BB,\SSS,\sigma]]
\ge
\frac{\Vert S_B(g)\Vert_{\BB,\SSS,\sigma}}{\Vert g\Vert_{\BB,\SSS,\sigma}} 
=
\frac{\Vert S_A(\sum_{n=1}^r a_n\xx_n)\Vert_\XX}{\Vert \sum_{n=1}^r a_n\xx_n\Vert_\XX}.
\]
We finish the proof by taking the supremum on $(a_n)_{n=1}^r$.
\end{proof}

\begin{proposition}\label{ConditionalityConstantsRelationTwo}
Let $\BB$  be a basis for a Banach space $\XX$, let
$\SSS$ be a subsymmetric sequence space and  $\sigma=(\sigma_n)_{n=1}^\infty$ be an ordered partition of $\NN$. Assume that $L_m[\BB]\gtrsim \delta(m)$ for $m\in\NN$ for some non-decreasing doubling function  $\delta\colon[0,\infty)\to[0,\infty)$ and that
\begin{equation}\label{conditionB}
\log\left(\sum_{n=1}^r |\sigma_n|\right) \lesssim r  \text{  for } r\in\NN.
\end{equation}
Then
$
L_{m}[\EE,\YY[\BB,\SSS,\sigma]]\gtrsim \delta(\log m)
$
for $m\in\NN$.
\end{proposition}
\begin{proof}
If $M_r=\sum_{n=1}^r |\sigma_n|$, there is  $C_1>0$ such that $\log M_r \le C_1 (r-1)$ for all $r\ge 2$.
Let  $C_2\in(0,\infty)$ be such that $\delta(C_1 x)\le C_2 \delta(x)$ for all $x\ge 0$. Let $C_3\in(0,\infty)$ be such that $\delta(m)\le C_3 L_m[\BB]$ for all $m\in\NN$.
Given $m\in\NN$ with $m\ge M_1$, pick $r\ge 1$ such that $M_r\le m <M_{r+1}$. 
Invoking Lemma~\ref{ConditionalityConstantsRelation} we get
\begin{align*}
C_2 C_3 L_{m}[\EE,\YY[\BB,\SSS,\sigma]]&
\ge C_2 C_3  L_{M_r}[\EE,\YY[\BB,\SSS,\sigma]]\\
&\ge C_2 C_3 L_r[\BB,\XX]\\
&\ge C_2 \delta(r)\\
&\ge  \delta(C_1 r)\\
&\ge \delta(\log M_{r+1})\\
&\ge \delta(\log m).
\end{align*}
\end{proof}

\subsection{Democracy-like properties  from the DKK-method}\label{DKKDemocratic}
\smallskip

\noindent The super-democracy of the unit-vector system in sequence spaces  obtained by the DKK-method will be inferred by means of embeddings
involving Lorentz sequence spaces. 

Given a weight $\ww=(w_n)_{n=1}^\infty$, the weak Lorentz sequence space
 $d_1^\infty(\ww)$ consists of all sequences $f=(a_n)_{n=1}^\infty\in c_0$ whose non-increasing rearrangement $(a_n^*)_{n=1}^\infty$ verifies
\[
\textstyle
\Vert f\Vert_{d_1^\infty(\ww)}=\sup_m a_m^* \sum_{n=1}^m w_n
=\sup_{a>0} a \sum_{n=1}^{|\{j\colon |a_j|>a\}|} w_n
  <\infty.
\]
It is  well-known that an embedding of the form \[
d_1(\ww)\subseteq \YY \subseteq d_1^\infty(\ww)\] implies that the unit-vector basis of the sequence space $\YY$ is super-democratic with fundamental function  equivalent to $(\sum_{n=1}^m w_n)_{m=1}^\infty$.
Conversely, any almost greedy basis fulfils  such embeddings  (see \cite{AA2015}*{Theorem 3.1}). 

In the case when  $(\SSS,\Vert \cdot\Vert_\SSS)$ is a subsymmetric sequence space, the sequence $(\Lambda_n)_{n=1}^\infty$ defined as in \eqref{BOSforAP}
 is the fundamental function of the unit-vector system and if $\ww
=(\Lambda_n-\Lambda_{n-1})_{n=1}^\infty$ then
\begin{equation}\label{EmbeddingSubSym}
\Vert f\Vert_{d_1^\infty(\ww)} \le \Vert f\Vert_\SSS\le \Vert f\Vert_{d_1(\ww)}
\end{equation}
(cf.\ \cite{AADK2016}*{Theorem 6.1}).
Note  
 that while $(\Lambda_n)_{n=1}^\infty$ is non-decreasing,  the sequence $(\Lambda_n/n)_{n=1}^\infty$ is non-increasing (see \cite{DKKT2003}*{comments below Theorem 3.1}). Of course, this monotonicity can be expressed as
\begin{equation}\label{Monotony}
\frac{\Lambda_m}{m}\le \frac{\Lambda_n}{n}, \quad m\ge n,
\end{equation}
but also in the form
\begin{equation}\label{WeightsIneq}
\Lambda_n-\Lambda_{n-1}\le \frac{\Lambda_n}{n}, \quad n\in\NN.
\end{equation}
\begin{theorem}[cf. \cite{DKK2003}*{Proposition 6.1}]\label{EmbeddingTheorem} Let $\BB$ be a  basis for a Banach space $\XX$, let $(\SSS,\Vert \cdot\Vert_\SSS)$  be a subsymmetric sequence space, and
 $\sigma=(\sigma_n)_{n=1}^\infty$ be an ordered partition of $\NN$. Assume that 
\begin{equation}\label{conditionA}
M_r:=\sum_{n=1}^{r} |\sigma_n| \lesssim |\sigma_{r+1}| \text{   for  } r\in\NN.
\end{equation}
 Let  $\ww=(\Lambda_n - \Lambda_{n-1})_{n=1}^\infty$ be as in \eqref{BOSforAP}
and $\ww'=(\Lambda_n / n)_{n=1}^\infty$.   Then 
\[
d_1(\ww')\subseteq \YY[\BB,\SSS,\sigma]\subseteq d_{1}^\infty(\ww)
\]
(with continuous embeddings).
\end{theorem}
\begin{proof} Let $C_\sigma=\sup_r M_r/ |\sigma_{r}|$. Assume that $\BB=(\xx_n)_{n=1}^\infty$ is a  bi-monotone and normalized
 basis for $(\XX,\Vert \cdot\Vert_\XX)$. 

Note that if $j\in\sigma_n$ then, by \eqref{Monotony},
\[
\frac{\Lambda_j}{j}\ge  \frac{\Lambda_{M_n}}{M_n} \ge
 \frac {\Lambda_{|\sigma_n|}}{ C_\sigma |\sigma_n|}.
\]
If  $f=(a_j)_{j=1}^\infty\in c_{00}$,
 \begin{align*}
\left\Vert \sum_{n=1}^\infty \vv_n^*(f) \xx_n \right\Vert_\XX &
\le \sum_{n=1}^\infty | \vv_n^*(f) |   \\
&\le \sum_{n=1}^\infty \frac{\Lambda_{|\sigma_n|}}{|\sigma_n|} \sum_{j \in \sigma_n} |a_j|\\
&\le C_\sigma\sum_{j=1}^\infty  \frac{\Lambda_j}{j} |a_j|.
\end{align*}
Hence, appealing to the rearrangement inequality and to \eqref{Monotony}, 
\[
\left\Vert \sum_{n=1}^\infty \vv_n^*(f) \xx_n\right\Vert_\XX\le  C_\sigma \Vert f\Vert_{d_1(\ww')}.\]
 By \eqref{ComplementAveragingSubsym}, \eqref{EmbeddingSubSym}
and \eqref{WeightsIneq}, we also have
 \[
\Vert Q_\sigma(f)\Vert_\SSS\le 3 \Vert f\Vert_\SSS\le 3 \Vert f\Vert_{d_1(\ww)}
 \le 3 \Vert f\Vert_{d_1(\ww')}.
\]
Combining, we get 
\[
\Vert f\Vert_{\BB,\XX,\sigma}\le (3+C_\sigma) \Vert f\Vert_{d_1(\ww')}.\]

Let us now look at the lower estimate. Suppose $a > 0$ and let $B := \{ k\in\NN \colon |a_k| > a \}$. Put $t=C_\sigma/(1+C_\sigma)$ and define
\[
A =
\{ n\in\NN \colon |\vv_n^*(f)| < ta  \Lambda_{|\sigma_n|}  \}
= \left\{ n\in\NN \colon
|\Av(f,\sigma_n)|
< ta
 \right\}.
\]
 If $k\in B\cap \sigma_n$ and $n\in A$, we have 
\[
|\ee_k^*(Q_\sigma(f))|=
\left| a_k- 
|\Av(f,\sigma_n)|
\right|>(1-t)a=
\frac{1}{1+C_\sigma} a.
\]
Consequently, if $B_0:=B \cap (\cup_{n \in A} \sigma_n)$, we have
\[
\Vert Q_\sigma(f)\Vert_\SSS\ge \frac{ a \Lambda_{|B_0|}}{1+C_\sigma}.
\]
If $m$ is the largest integer in $\NN\setminus A$
we have  $B_1:=B \setminus (\cup_{n \in A} \sigma_n)\subseteq\cup_{n=1}^m \sigma_n$. Therefore $|B_1|\le C_\sigma|\sigma_m|$ and then there is an integer $N$
with $\max\{ |B_1|,|\sigma_m|\}\le N\le  C_\sigma|\sigma_m|$.
 By \eqref{Monotony},
\[
\Lambda_{|B_1|}
\le
 \Lambda_N 
\le 
 \frac{N}{|\sigma_m|}    \Lambda_{|\sigma_m|}
\le \frac{C_\sigma}{ta} | \vv_m^*(f)|
= \frac{1+C_\sigma}{a} | \vv_m^*(f)|.
\]
Combining,
\[
\Vert f\Vert_{\BB,\SSS,\sigma}
\ge \Vert Q_\sigma(f)\Vert_\SSS + | \vv_m^*(f)|
\ge \frac{a}{1+C_\sigma} (\Lambda_{|B_0|}+\Lambda_{|B_1|})
\ge \frac{a}{1+C_\sigma} \Lambda_{|B|}.
\]
Hence
\[ \Vert f\Vert_{d_1^\infty(\ww)}\le (1+C_\sigma) \|f \|_{\BB,\SSS,\sigma}.\]
\end{proof}

Following \cite{DKKT2003}, we say that a weight $(\lambda_m)_{m=1}^\infty$  has the \textit{lower regularity property} (LRP for short) if there a positive integer $b$ such 
 \[
 2\lambda_m \le \lambda_{bm}, 
\quad
 m\in\NN.
 \] 
 We will also need the so-called \textit{upper regularity property} (URP for short). We say that $(\lambda_m)_{m=1}^\infty$ has the URP if there is an integer $s\ge 3$ such that
\[
 \lambda_{b m}\le \frac{b}{2} \lambda_m,\quad m\in\NN.
\] 
 \begin{corollary}\label{EmbeddingTwo}Assume that  the hypotheses of Theorem~\ref{EmbeddingTheorem} hold  and that  $(\Lambda_n)_{n=1}^\infty$ has the LRP.  Then 
 \[d_1(\ww)\subseteq \YY[\BB,\SSS,\sigma]\subseteq d_{1}^{\infty}(\ww).
\]
\end{corollary}
 
\begin{proof} From the  LRP and \eqref{Monotony}  (see, e.g., \cite{AA2015}*{Lemma 2.12}) we infer that $(\Lambda_n)_{n=1}^\infty$ satisfies  the Dini condition 
 \begin{equation*}
  \sum_{n=1}^m \frac{\Lambda_n}{n}\le C_d \Lambda_m \text{ for all }m\in\NN
 \end{equation*} 
 for some constant $C_d$. Consequently, $d_1(\ww)\subseteq d_1(\ww')$.
Combining with Theorem~\ref{EmbeddingTheorem} yields
\[d_1(\ww)\subseteq \YY[\BB,\SSS,\sigma] \subseteq d_1^\infty(\ww).\]
Quantitatively, if $C_\sigma=\sup_r M_r/ |\sigma_{r}|$,
\[
\frac{1}{1+C_\sigma}\Vert f\Vert_{d_1^\infty(\ww)}\le \Vert f \Vert_{\BB,\SSS,\sigma} \le  (3+C_\sigma) C_d \Vert f\Vert_{d_1(\ww)}.
\]
for all $f\in c_{00}$.
\end{proof}

\subsection{Quasi-greediness in bases  from the DKK-method}\label{DKKQG}

\smallskip
\noindent The embeddings provided by
Lemma~\ref{EmbeddingTwo} are considered for some authors as a property which  ensures in some sense the optimality of the estimation and compression
algorithms with respect to the basis (see \cite{Donoho}). The authors of \cite{AA2015} point out that almost greediness
is a stronger condition.
In this section we show that an extra property on the symmetric sequence space, namely URP, fills the gap between those two properties. 

\begin{lemma}\label{lem: projectionlemma} 
Let $(\SSS,\Vert \cdot\Vert_\SSS)$ be a subsymmetric sequence space and $\sigma=(\sigma_n)_{n=1}^\infty$ be an ordered partition of 
$\NN$. Let $\Lambda_n^*$ and $\vv_n^*$ for $n\in \NN$  be as in \eqref{DemocracySubSym} and \eqref{BOSforAP}. Then
\[
|\vv_n^*(S_A(f))|\le 2 \frac{\Lambda^*_{|A|}}{\Lambda^*_{|\sigma_n|}} \Vert S_{\sigma_n}(f) \Vert_\SSS
\]
whenever $n\in\NN$, $A\subseteq \sigma_n$ and $f\in c_{00}$.
\end{lemma}
\begin{proof}Assume without lost of generality that $\supp(f)\subseteq\sigma_n$.
Let $v^*=\sum_{j\in A} \ee_j^*$.  Taking into account \eqref{Monotony} and \eqref{bidemocracy} we obtain
\[
|\vv_n^*(S_A(f))| =\frac{\Lambda_{|\sigma_n|}}{|\sigma_n|} | v^*( f ) |
\le  \frac{\Lambda_{|\sigma_n|}}{|\sigma_n|}   \Vert v^*\Vert_{\SSS^*} \Vert  f  \Vert_\SSS
 \le 2 \frac{\Lambda_{|\sigma_n|}}{|\sigma_n|}  \frac{|A|}{\Lambda_{|A|}}   \Vert   f  \Vert_\SSS,
\]
as desired.
 \end{proof}

For further reference,  let us write down the following easy result.
\begin{lemma}\label{lem: projectionlemmal1} 
 Let $\sigma=(\sigma_n)_{n=1}^\infty$ be an ordered partition of $\NN$. Let $\Lambda_n$ and $\vv_n^*$ for $n\in \NN$  be as in \eqref{BOSforAP} with $\SSS=\ell_1$. Then
\[
|\vv_n^*(S_A(f))|
\le 
\frac{\Lambda_{|A|}}{\Lambda_{|\sigma_n|}}  |\vv_n^*(f)|
+
\sum_{j\in \sigma_n} |a_j -\Av(f,\sigma_n)|
\]
whenever $n\in\NN$, $A\subseteq \sigma_n$, and $f=(a_j)_{j=1}^\infty\in c_{00}$.
\end{lemma}
\begin{proof}We have $\Lambda_m=m$ for all $m$. Hence
\begin{align*}
|\vv_n^*(S_A(f))| & =\left| \sum_{j\in A}  a_j \right| \\
&\le |A|\,  |\Av(f,\sigma_n)|  + \left| \sum_{j\in A}  a_j -\Av(f,\sigma_n) \right| \\
&\le |A| \, |\Av(f,\sigma_n)| + \sum_{j\in \sigma_n}  |a_j -\Av(f,\sigma_n) | \\
& =\frac{ \Lambda_{|A|}}{\Lambda_{|\sigma_n|}}  |\vv_n^*(f)| + \sum_{j\in \sigma_n}  |a_j -\Av(f,\sigma_n) | \\
\end{align*}
\end{proof}

\begin{lemma}\label{lemma:7}
 Let $(\lambda_n)_{n=1}^\infty$ be a non-decreasing sequence of positive scalars such that $(\lambda_n/n)_{n=1}^\infty$ is non-increasing. Then
 \[
 \frac{k}{k+n}\lambda_n +\frac{n}{k+n}\lambda_k\le 2 \lambda_k
 \]
 for all $k$, $n\in\NN$.
  \end{lemma}
  \begin{proof}In the case when $k\le n$ we have
  \[
   \frac{k}{k+n}\lambda_n +\frac{n}{k+n}\lambda_k\le  \frac{n}{k+n}\lambda_k + \frac{n}{k+n}\lambda_k =\frac{2n}{n+k}  \lambda_k\le 2 \lambda_k,
   \]
  and in the case $n\le k$,
  \[
   \frac{k}{k+n}\lambda_n +\frac{n}{k+n}\lambda_k\le  \frac{k}{k+n}\lambda_k +\frac{n}{k+n}\lambda_k = \lambda_k.
   \]
  \end{proof}

\begin{lemma}\label{lem: Axnormestimate}Let $\SSS$ be a subsymmetric sequence space and 
$\sigma=(\sigma_n)_{n=1}^\infty$ be an ordered partition of $\NN$. 
  For all $A \subseteq\NN $ and $f \in\SSS $ we have
$$ \|Q_\sigma(S_A(f))\|_\SSS  \le 5 \|Q_\sigma(f)\|_\SSS + 2 \sum_{n= 1}^\infty \frac{\Lambda_{|A_n|}} {\Lambda_{|\sigma_n|}} |\vv_n^*(f)|,$$ 
where $A_n := A \cap \sigma_n$.
\end{lemma}
\begin{proof} We write $\|\cdot\|$ rather than $\|\cdot\|_\SSS$ as there is no possibility of confusion in this proof.
Put $B_n=\sigma_n\setminus A_n$. 
For every $f\in Q_\sigma(\SSS)$ note that 
\[
\Av(f,A_n) |A_n| +\Av(f,B_n) |B_n| =0
\]
and, hence,
$
P_\sigma(S_A(f))=(y_j)_{j=1}^\infty,
$
where \[y_j=\begin{cases}\displaystyle  \Av(f,A_n) \frac{|A_n|} {|\sigma_n|} & \text{ if } j\in A_n \\ 
- \displaystyle \Av(f,B_n) \frac{|B_n|}{|\sigma_n|} & \text{ if } j\in B_n.
\end{cases}\]

Let $P_{\tau}$ be the averaging projection with respect to the partition
$\tau=(A_n,B_n)_{n=1}^\infty$. We infer from Theorem~\ref{AveragingSubsymmetric}
that $\Vert P_{\tau}\Vert_{\SSS\to\SSS}\le 4$. Taking into account  the lattice structure  of $\SSS$ we get
\[
\Vert Q_\sigma(S_A(f))\Vert 
\le 
\Vert P_\sigma(S_A(f))\Vert + \Vert S_A(f)\Vert
 \le \Vert P_{\tau}(f)\Vert +\Vert f\Vert 
\le 5 \Vert f\Vert.
\]
Assume that $f\in P_\sigma(\SSS)$. 
Pick $(b_n)_{n=1}^\infty$ such that $\ee_j^*(f)=b_n$ if $j\in\sigma_n$. 
Then
$Q_\sigma(S_A(f))=(c_j)_{j=1}^\infty$, where 
\[
c_j=\begin{cases}\displaystyle b_n  \frac{|B_n|}{|\sigma_n|} & \text{ if } j\in A_n \\ 
\displaystyle-b_n \frac{|A_n|}{|\sigma_n|} & \text{ if } j\in B_n.
\end{cases}
\]
Therefore, taking into account  Lemma~\ref{lemma:7},
\begin{align*}
\Vert Q_\sigma(S_A(f))\Vert 
&\le\sum_{n=1}^\infty |b_n| 
\left( \frac{|B_n|}{|\sigma_n|}  \Lambda_{|A_n|} + \frac{|A_n|}{|\sigma_n|}  \Lambda_{|B_n|}\right)\\
&\le 2\sum_{n=1}^\infty |b_n|  \Lambda_{|A_n|}\\
&= 2\sum_{n=1}^\infty   \frac{\Lambda_{|A_n|}}{\Lambda_{|\sigma_n|}} | \vv_n^*(f)|.
\end{align*}
We complete the proof by expressing any $f\in\SSS$ in the form $f_1+f_2$ with $f_1\in P_\sigma(\SSS)$ and
$f_2\in Q_\sigma(\SSS)$, and combining.
\end{proof}

\begin{lemma}\label{lem: regularityrstimate} 
Let $(N_n)_{n=1}^\infty$ a sequence  of positive integers such that $\sum_{k=1}^n  N_k \lesssim N_{n+1}$ for $n\in\NN$, and let $(\lambda_n)_{n=1}^\infty$ be a sequence of positive scalars having  the LRP. Assume that
$(\lambda_n)_{n=1}^\infty$  and  $(n/\lambda_n)_{n=1}^\infty$ are non-decreasing. Then
\[ 
\sup_{r\in\NN} \sum_{n=r}^\infty \frac{\lambda_{N_r}}{\lambda_{N_n}} <\infty.
\]
\end{lemma}
 \begin{proof}
 Since $(\lambda_n)_{n=1}^\infty$ has the LRP,
appealing to \cite{AA2015}*{Theorem 2.12}
 we claim the existence of $0<\alpha<1$ and $0<C_1<\infty$ such that
\[
\frac{\lambda_n}{n^\alpha} \le C_1 \frac{\lambda_m}{m^\alpha} \text{ for }n\le m.
\]
Let $C_2>1$ be such that
$M_r:=\sum_{n=1}^r N_n\le C_2 N_r$ for all $r\in\NN$.  Then, if $t=C_2/(1+C_2)$, we have
 $M_{r+1}\le t  M_r $ for $r\in\NN$.
Hence
\begin{align*}
 \sum_{n=r}^\infty \frac{\lambda_{N_r}}{\lambda_{N_n}}
& \le  \sum_{j=n}^\infty\frac{M_n}{N_n} \frac{\lambda_{N_r}}{\lambda_{M_n}}
  \le  C_2 \sum_{n=r}^\infty  \frac{\lambda_{M_r}}{\lambda_{M_n}}
\le C_1 C_2 \sum_{n=r}^\infty \left(\frac{M_r}{M_n}\right)^\alpha\\
&\le  C_1 C_2 \sum_{n=r}^\infty  t^{\alpha(n-r)}
= C_1 C_2 \frac{1}{1-t^\alpha},
\end{align*}
as desired.
\end{proof}
\begin{lemma}\label{lem: projectionnormestimate} 
Assume that all the hypotheses of Theorem~\ref{EmbeddingTheorem} hold  and that 
either $\SSS=\ell_1$ or $(\Lambda_n)_{n=1}^\infty$ has both 
 LRP and the URP. Then there exists a constant $C_a$ such that  whenever $A$ and $r$ are such that $A \subset \cup_{n= r}^\infty \sigma_n$
 and $|A| \le M_r$,
we have $\Vert S_A(f)\Vert_{\BB,\SSS,\sigma}
 \le  C_a\Vert f\Vert_{\BB,\SSS,\sigma}$ for all $f\in c_{00}$.\end{lemma}
\begin{proof} 
Without loss of generality assume that $\BB=(\xx_n)_{n=1}^\infty$ is  normalized and bimonotone.
Put $A_n=A\cap \sigma_n$. By assumption there is $r\in\NN$ such that $A_n= \emptyset$ for $1 \le n\le r-1$ and, then,
\[
\left\|\sum_{n=1}^\infty \vv_n^*(S_A(f)) \, \xx_n \right\|_\XX 
\le \sum_{n=1}^\infty  |\vv_n^*(S_A(f))|
=\sum_{n=r}^\infty  |\vv_n^*(S_{A_n}(f))|.
\]
 In the case when $(\Lambda_n)_{n=1}^\infty$ has the URP, we infer from
 inequality \eqref{Monotony} and \cite{AA2015}*{Theorem 2.12} that $(\Lambda^*_n)_{n=1}^\infty$  has the LRP.
Using the fact that $(\Lambda^*_n)_{n=1}^\infty$ is doubling and Lemma~\ref{lem: regularityrstimate} gives
\[
\quad C_1:=\sup_{r}  \frac{\Lambda^*_{M_r}}{\Lambda^*_{|\sigma_r|}}<\infty, \quad
C_2=\sup_r \sum_{n=r}^\infty  \frac{\Lambda^*_{|\sigma_r|}}{ \Lambda^*_{|\sigma_n|}}<\infty.
 \]
 Then, by Lemma~\ref{lem: projectionlemma}, Lemma~\ref{lem: normequivalence5} and Lemma~\ref{CommutingLemma},
 \begin{align*}
\sum_{n=r}^\infty  |\vv_n^*(S_{A_n}(f))|
 &\le 2 \sum_{n=r}^\infty \frac {\Lambda^*_{|A_n|}} {\Lambda^*_{|\sigma_n|}}  \Vert S_{\sigma_n}(f)\Vert_\SSS\\
&\le 2 \sum_{n=r}^\infty \frac {\Lambda^*_{M_r}} {\Lambda^*_{|\sigma_n|}}\Vert S_{\sigma_n}(f) \Vert_{\BB,\SSS,\sigma} \\
& \le 2 C_1 \Vert f \Vert_{\BB,\SSS,\sigma}  \sum_{n=r}^\infty \frac {\Lambda^*_{|\sigma_r|}}{\Lambda^*_{|\sigma_n|}}\\
& \le 2 C_1 C_2 \Vert f \Vert_{\BB,\SSS,\sigma}.
\end{align*}  
In the case when $\SSS=\ell_1$, Lemma~\ref{lem: projectionlemmal1} gives
 \begin{align*}
 \sum_{n=r}^\infty  |\vv_n^*(S_{A_n}(f))|
&\le
 \sum_{n=r}^\infty \frac{\Lambda_{|A_n|}}{\Lambda_{|\sigma_n|}} |\vv_n^*(f) |
+\sum_{n=r}^\infty \sum_{j\in \sigma_n} |a_j-\Av(f,\sigma_n)|\\
&\le \Vert Q_\sigma(f)\Vert_1+ \sum_{n=r}^\infty \frac{\Lambda_{|A_n|}}{\Lambda_{|\sigma_n|}} |\vv_n^*(f) |
  \end{align*}
The sequence $(\Lambda_n)_{n=1}^\infty$ is also doubling, and
our hypothesis always gives that it has the LRP. Hence, appealing again to Lemma~\ref{lem: regularityrstimate},
 \[
 C_3:=\sup_{r}  \frac{\Lambda_{M_r}}{\Lambda_{|\sigma_r|}}<\infty,\quad
   C_4=\sup_r \sum_{n=r}^\infty  \frac{\Lambda_{|\sigma_r|}}{ \Lambda_{|\sigma_n|}}<\infty.
  \]
Consequently,
 \begin{align*}
\sum_{n=1}^\infty \frac{\Lambda_{|A_n|}}{\Lambda_{|\sigma_n|}}|\vv_n^*(f) |
&\le \sum_{n=r}^\infty \frac{\Lambda_{M_r}}{\Lambda_{|\sigma_n|}}|\vv_n^*(f)|\\
&\le C_3 \sum_{n=r}^\infty \frac{\Lambda_{|\sigma_r|}}{\Lambda_{|\sigma_n|}} |\vv_n^*(f) | \\
& \le C_3 C_4 \left\Vert \sum_{n=1}^\infty  \vv_n^*(f) \, \xx_n\right\Vert_\XX.
\end{align*}
Using Lemma~\ref{lem: Axnormestimate} and combining we get the desired result with $C_a=
    2 C_1 C_2+\max\{5, 2C_3 C_4\}$ in the case when $(\Lambda_n)$ has both the LRP and the URP,
 and 
 $C_a=\max\{6, 3 C_1 C_2\}$ when $\SSS=\ell_1$.
 \end{proof}

\begin{theorem}[cf.\ \cite{DKK2003}*{Theorem 7.1}]
\label{QGTheorem}Assume that all the hypotheses of Theorem~\ref{EmbeddingTheorem} hold and that either $\SSS=\ell_1$ or  $(\Lambda_n)_{n=1}^\infty$  has both the LRP and the URP. Then the unit-vector system is an almost greedy basis for $\YY[\BB,\SSS,\sigma]$ with fundamental function equivalent  to  $(\Lambda_n)_{n=1}^\infty$.
\end{theorem}

\begin{proof}
Assume that $\BB$ is bi-monotone. By  Corollary~\ref{EmbeddingTwo}, which asserts that for all $f\in c_{00}$ and some positive constants $C_{1}, C_{2}$,
\[
\frac{1}{C_{1}}\Vert f\Vert_{d_1^\infty(\ww)}\le \Vert f \Vert_{\BB,\SSS,\sigma} \le  C_{2}\Vert f\Vert_{d_1(\ww)},
\]
 it suffices to prove that
$\EE$ is a quasi-greedy basis for  $\YY:=\YY[\BB,\SSS,\sigma]$. Let $C_a$ be as in 
Lemma~\ref{lem: projectionnormestimate}. By  Theorem~\ref{FirstProperties} (a),
 \[
C_b:=\sup_{m\in\NN} \Vert \Id_{\FF^\NN} -S_m\Vert_{\YY\to\YY}<\infty.
\]

Let $f=(a_j)_{n=1}^\infty\in c_{00}$ and let $F\subseteq\NN$ be a non-empty set such that $|a_j|\le |a_k|$ whenever $k\in F$ and $j\in\NN\setminus F$. Denote $m=|F|$ and pick
$r\in\NN$ such that  $m \in \sigma_r$. Let $A=[1,m]\setminus F$ and
$B=  F\cap [m+1,\infty)$. We have
\[
F\cup A=\{1,\dots,m\}\cup B, \quad F\cap A=\{1,\dots,m\}\cap B=\emptyset.
\]
Therefore
$s:=|A|=|B|\le m \le M_r$,  $B\subseteq \cup_{n = r}^\infty \sigma_n$ and 
 \[
S_F(f)=S_m(f)-S_A(f)+S_B(f).
\]
We infer that
$
\Vert S_B(f)\Vert_{\BB,\SSS,\sigma}\le C_a \Vert f\Vert_{\BB,\SSS,\sigma}
$ and that,
if $(a_n^*)_{n=1}^\infty$ is the non-increasing rearrangement of $f$, $|a_j|\le a_s^*$ for all $j\in A$.
Then, 
\[
\Vert S_A(f)\Vert_{\BB,\SSS,\sigma} \le C_{2} \max_{j\in A}|a_j| \Lambda_s\\
\le a_s^* \Lambda_r\\
\le 
C_{1}C_{2} \Vert f\Vert_{\BB,\SSS,\sigma}.
\]
Combining we get 
\begin{align*}
\Vert f -S_F(f)\Vert_{\BB,\SSS,\sigma}
&\le \Vert f -S_{k}(f)\Vert_{\BB,\SSS,\sigma}
+\Vert S_A(f)\Vert_{\BB,\SSS,\sigma}
+\Vert S_B(f)\Vert_{\BB,\SSS,\sigma}\\
& \le (C_b + C_a +C_1 C_2) \Vert f\Vert_{\BB,\SSS,\sigma}.
\end{align*}
That is, the unit vector system is  $(C_b + C_a +C_1 C_2)$-quasi-greedy.
\end{proof}

\begin{remark} Note that $\eqref{conditionA}$ implies
\begin{equation}\label{Neweq}
\log\left(\sum_{n=1}^{r}|\sigma_{n}| \right)\gtrsim r,\quad r\in \NN.
\end{equation}
This exponential growth is  essentially optimal. Indeed, Theorem~\ref{CharacterizationSR}(a) implies that   when the DKK-method is applied to a basis $\BB$ for a Banach space $\XX$ with  $L_m[\BB, \XX] \approx m$ (such as  the summing basis of $c_0$) it can only  produce a Banach space $\YY[\BB,\SSS,\sigma]$ for which the  unit-vector basis  is quasi-greedy  when
\eqref{Neweq} holds.
\end{remark}

\section{Banach spaces having quasi-greedy bases with large conditionality constants}
\label{Main}

\noindent The conductive thread of this section is the search for results that will allow us to include the spaces $Z_{p,q}$, $B_{p,q}$, and $D_{p,q}$ (see Section~\ref{Introduction}) in the list of Banach spaces possessing highly conditional quasi-greedy bases.
We recall that the matrix spaces $Z_{p,q}$ are isomorphic to Besov spaces over Euclidean spaces (see, e.g., \cite{AA2016}) and that the  mixed-norm spaces $B_{p,q}$ are isomorphic to Besov spaces over the unit interval (see, e.g., \cite{AA2017}*{Appendix 4.2}). 

Apart from the trivial cases,  namely 
\[D_{q,q}\approx Z_{q,q}\approx B_{q,q}\approx \ell_q, \quad  1\le q<\infty,\] and the case \begin{equation}\label{Peuchinski}
\ell_q \approx B_{2,q}, \quad 1<q<\infty,
\end{equation}
all the above-mentioned    spaces are mutually non-isomorphic (see \cite{AA2017}).
 The isomorphism in  \eqref{Peuchinski} was obtained by Pe{\l}czy{\'n}ski  in \cite{Pel1960}  by combining the uniform complemented embeddings
\begin{equation}\label{PeuRad}
\ell_2^n \lesssim_c \ell_p^{2^n}  \text{ for } n\in\NN \text{, if } 1<p<\infty,
\end{equation}
 (which can  be obtained as a consequence of the boundedness of the Rademacher projections in  $L_p$) 
 with the Pe{\l}czy{\'n}ski  decomposition technique (see, e.g., \cite{AlbiacKalton2016}*{Theorem 2.2.3}).
Another well-known consequence  of Pe{\l}czy{\'n}ski  decomposition technique, 
is that for any unbounded sequence of integers $(d_n)_{n=1}^\infty$   we have
\begin{equation}\label{BesovIso}
B_{p,q}\approx(\oplus_{n=1}^\infty \ell_p^{d_n})_q, \quad p\in [1,\infty], \, q\in\{0\}\cup[1,\infty),
\end{equation}
(see, e.g.,  \cite{AA2017}*{Appendix 4.1}.)

\begin{theorem}\label{MainDKKTheorem}Let $\XX$ be a Banach space with a basis $\BB$ and suppose that either $\SSS=\ell_1$ or 
$\SSS$ is a subsymmetric sequence space with nontrivial type.
 Assume that  $\SSS\lesssim_c \XX$ ant that $L_m[\BB]\gtrsim \delta(m)$ for $m\in\NN$ for some doubling non-decreasing function $\delta\colon[0,\infty)\to[0,\infty)$.  
Then there is an almost greedy basis $\BB_{\kappa}$ for $\XX$ 
with fundamental function equivalent to   $(\Vert\sum_{j=1}^n \ee_j\Vert_\SSS)_{n=1}^{\infty}$, 
 and with
$L_m[\BB_{\kappa}]\gtrsim \delta(\log m)$ for $m\in\NN$.
\end{theorem}

\begin{proof}For each $n\in\NN$ put $\sigma_n=[2^{n-1},2^n-1]$, and let $\vv_n$ be defined as in \eqref{BOSforAP}. Notice  that  $\sigma=(\sigma_n)_{n=1}^\infty$ verifies both \eqref{conditionA} and \eqref{conditionB} and that $\VB=(\vv_n)_{n=1}^\infty$ is an unconditional basis for $P_\sigma[\SSS]$.   Then,
by \cite{DKKT2003}*{Proposition 4.1}, Lemma~\ref{LemmaOne},  Theorem~\ref{QGTheorem} 
 and Proposition~\ref{ConditionalityConstantsRelationTwo},
the unit-vector system is an almost greedy basis as desired for $\YY:=\YY[\VB\oplus\BB,\SSS,\sigma]$. Let $\ZZ$ such that $\XX\approx \SSS\oplus\ZZ$.
The chain of isomorphisms
\[
\YY\approx  Q_\sigma(\SSS)\oplus P_\sigma(\SSS)\oplus\XX
\approx \SSS\oplus \SSS\oplus\ZZ
\approx \SSS\oplus\ZZ
\approx\XX
\]
(see Remark~\ref{SubSymSquare}) completes the proof.
\end{proof}

In order to effectively use Theorem~\ref{MainDKKTheorem} we need to ensure the existence of bases $\BB$  with large conditionality constants $(L_m[\BB])_{m=1}^{\infty}$ in subsymmetric sequence spaces. 
We start by recalling the following result by Garrig\'os and Wojtaszczyk, which, in our language, reads as follows.
\begin{theorem}[cf.\ \cite{GW2014}*{Proposition 3.10}]\label{GWHilbert} For each $0<a<1$ there is a  basis $\BB$ in $\ell_2$ with $L_m[\BB]\gtrsim m^a$ for $m\in\NN$.
\end{theorem}

\begin{proposition}\label{HCLp}
For each $0<a<1$ and each $1<q<\infty$ there is a  basis $\BB$ in $\ell_q$ with $L_m[\BB]\gtrsim m^a$ for $m\in\NN$.
\end{proposition}

\begin{proof} Apply Theorem~\ref{GWHilbert} for picking a basis $\BB_{0}=(\xx_j)_{j=1}^\infty$ for $\ell_2$ with $L_m[\BB_{0}]\gtrsim m^a$ for $m\in\NN$. By Lemma~\ref{LemmaTwo},
$\BB=\bigoplus_{n=1}^\infty  (\xx_j)_{j=1}^{2^n}$ is a  basis for $\XX=(\bigoplus \ell_2^{(2^n)}[\BB_{0}])_p$ with $L_m[\BB]\gtrsim m^a$. Since any $N$-dimensional Hilbert space is isometric to $\ell_2^N$, the isomorphisms \eqref{Peuchinski} and \eqref{BesovIso} yield $\XX\approx (\bigoplus \ell_2^{2^n})_q\approx B_{2,q}\approx\ell_q$.
\end{proof}

Our next result improves Corollary 3.13 from \cite{GW2014}, where it is shown the existence of quasi-greedy bases as conditional as possible in $\ell_q$, $1<q<\infty$. The main improvement consists of building, for $q\not=2$, almost greedy bases instead of quasi-greedy ones.

\begin{theorem}[cf. \cite{GW2014}*{Theorem 1.2, Corollary 3.12 and Corollary 3.13}]\label{T30}
Let $\XX$ be a Banach space with a basis and $1<q<\infty$. If
$\ell_q\lesssim_c \XX$ then for any  $0<a<1$ the space $\XX$ has an almost greedy basis $\BB_{\kappa}$ with fundamental function  equivalent to $(m^{1/q})_{m=1}^\infty$ and  with $L_m[\BB_{\kappa}]\gtrsim(\log m)^a$ for $m\in\NN$.
\end{theorem}
\begin{proof} By Lemma~\ref{LemmaOne} and Proposition~\ref{HCLp},  $\ell_p\oplus\XX$
has a basis $\BB$ with $L_m[\BB] \gtrsim m^a$ for $m\in\NN$. Then, Theorem~\ref{MainDKKTheorem}
gives a basis as desired for $\ell_q\oplus\XX$. Finally, since $\ell_q\oplus\ell_q\approx\ell_q$, we infer that
$\ell_q\oplus\XX\approx \XX$.
\end{proof}

\begin{example}\label{Ex30}\
\begin{itemize}
\item[(i)] 
Theorem~\ref{T30} applies to $L_q$ and to $\ell_q$ for $1<q<\infty$. 
More generally,  Theorem~\ref{T30} yields that if $q\in(1,\infty)$
  any separable $\LL_q$-space $\XX$ has  almost greedy bases as conditional as possible whose fundamental function is equivalent to $(m^{1/q})_{m=1}^\infty$ (see \cite{LinPel1968}*{Proposition 7.3} and  \cite{JRZ1971}*{Theorem 5.1}). Moreover, if $\XX$
   is not isomorphic to $\ell_q$ it also has   almost greedy bases as conditional as possible whose fundamental function is equivalent to $(m^{1/2})_{m=1}^\infty$ (see \cite{JohnsonOdell1974}*{Corollary 1} and \cite{KP1962}*{Corollary 1}).
   
 \item[(ii)] Let $1<p,q<\infty$.
 Theorem~\ref{T30} applies to the superreflexive spaces  $Z_{p,q}$, $B_{p,q}$ and $D_{p,q}$. 
 
 \item[(iii)]  Theorem~\ref{T30} also applies to  
Lorentz sequence spaces. In fact, appealing also to
 Theorem~\ref{MainDKKTheorem}  we claim that, given $\ww=(w_n)_{n=1}^\infty$ non-increasing,  $1<q<\infty$, and $\epsilon>0$, the space
  $d_q(\ww)$ has almost greedy bases $\BB_{\kappa}$ with $L_{m}[\BB_{\kappa}]\gtrsim (\log m)^{1-\epsilon}$ and with fundamental function equivalent either to
 $(m^{1/q})_{m=1}^\infty$ or to $(W_m^{1/q})_{n=1}^\infty$, where
$W_m=\sum_{n=1}^m w_n$. In the case when $(W_m)_{m=1}^\infty$ has the LRP
$d_q(\ww)$ is superreflexive (see  \cite{Al75}*{Theorem 1}) and, then, this is sharp.  In particular, the classical Lorentz sequence spaces $\ell_{p,q}$  for $1< p,q<\infty$ have almost greedy bases as conditional as possible with fundamental function   equivalent either to  $(m^{1/q})_{m=1}^\infty$ or  to $(m^{1/p})_{m=1}^\infty$.
 \end{itemize}
\end{example}

Before returning to our main theme of almost greedy bases for concrete spaces let us  present a more abstract application of Theorem~\ref{MainDKKTheorem}.
\begin{theorem}  \label{subsymmbasicsequence} Let $\SSS$ be a subsymmetric sequence space with nontrivial type. Then, for every $0 < a < 1$,  $\SSS$ contains an almost greedy basic sequence
$\BB_{\kappa}=(\xx_n)_{n=1}^\infty$ such that $L_m[\BB_{\kappa}]\gtrsim(\log m)^a$ for $m\in\NN$.
\end{theorem} \begin{proof}  Combining the Mazur construction of basic sequences (see e.g., \cite{AAW2017}*{Proposition 3.1}) with Dvoretzky's theorem on
 the finite representability of $\ell_2$ in all infinite-dimensional Banach spaces \cite{D1960} it follows that  every infinite-dimensional Banach space contains a basic sequence $\BB = (\xx_j)_{j=1}^\infty$
such that $[\xx_j]_{j=2^n}^{j= 2^{n+1}-1}$ is $2$-isomorphic to $\ell_2^{2^{n}}$  for all $n \ge 1$.   Moreover, Theorem~\ref{GWHilbert}
allows us to select the basis vectors $[\xx_j]_{j=2^n}^{j= 2^{n+1}-1}$ in such a way that we may ensure that  $L_m[\BB]\gtrsim m^a$ for $m\in\NN$.
Choose such a basic sequence $\BB$ inside $\SSS$ and let $\XX$ be its closed linear span. Setting $\sigma=([2^{n-1},2^n-1])_{n=1}^\infty$,
by Theorem~\ref{MainDKKTheorem} the unit-vector basis $\EE$    is an almost greedy basis   
 for $\YY= \YY[\BB,\SSS,\sigma]$ satisfying $L_m[\EE, \YY]\gtrsim(\log m)^a$ for $m\in\NN$. Finally, by Theorem~\ref{FirstProperties}~(c)
 and Remark~\ref{SubSymSquare},
$$ \YY[\BB,\SSS,\sigma] \approx Q_\sigma(\SSS)\oplus \XX\subseteq \SSS \oplus \SSS \approx \SSS.$$
\end{proof}
\begin{remark} It follows from Theorem~\ref{subsymmbasicsequence} that if $\XX$ has nontrivial type then every spreading model for $\XX$ generated by a weakly null sequence contains a basic sequence satisfying the conclusion of Theorem~\ref{subsymmbasicsequence}. This follows because spreading models of $\XX$ are finitely represented in $\XX$ (and hence have nontrivial  type) and those generated by weakly null sequences are subsymmetric.
\end{remark}

Example~\ref{Ex30} shows that  Theorem~\ref{MainDKKTheorem} is strong enough for a  wide class of superreflexive Banach spaces. When dealing with non-superreflexive Banach spaces we need to combine the DKK-method with other techniques.
To that end, we recover some ideas from \cite{AAW2017}.
Recall that a basis $(\xx_j)_{j=1}^\infty$ is said  to be of \textit{type $P$} if   
\[\sup_k \left\Vert \sum_{j=1}^k \xx_j \right\Vert <\infty\quad \text{and}\quad  \inf_j \Vert \xx_j\Vert>0.\] 
Notice that both the unit-vector system in $c_0$ and the \textit{difference system}
$\DD=(\dd_j)_{j=1}^\infty$  in $\ell_1$, defined by
\[
\dd_j=\ee_j -\ee_{j-1}  \text{ (with the convention $\ee_0=0$)},
\]
 are bases of type P.
 It is known that the   \textit{summing system} $\SSB=(\sss_j)_{j=1}^\infty$, given by
\[
\quad \sss_j=\sum_{k=1}^j \ee_j.
\]
 is a conditional basis for $c_0$. Most proofs  of this fact 
(see, e.g., \cite{AlbiacKalton2016}*{Example 3.1.2}) give the following.
\begin{lemma}\label{coLemma}\label{Summingc0} The summing system $\SSB$ is a basis  for $c_0$ with  $L_m[\SSB,c_0]\approx m$ for $m\in\NN$, and $c_0^{(N)}[\SSB]=\ell_\infty^N$ for all $N\in\NN$.
\end{lemma}
By duality, the difference system is a conditional basis for $\ell_1$. Indeed, we have the following.
\begin{lemma}\label{l1Lemma} The difference system $\DD$ is a basis of $\ell_1$ with $L_m[\DD,\ell_1]\approx m$ for $m\in\NN$, and  $\ell_1^{(N)}[\DD]=\ell_1^N$ for all $N\in\NN$.
\end{lemma}
Lemmas~\ref{coLemma} and \ref{l1Lemma} exhibit the fact that $c_0$ and $\ell_1$ have bases as conditional as possible.
 The following lemma shows  that Banach spaces with a  basis of type P follow the pattern  of  $c_0$ and $\ell_1$.
 
  \begin{lemma}[see \cite{AAW2017}]\label{AnsoLemma}Let $\BB_{0}$ be a basis of type $P$ of a Banach space $\XX$. Then there is a basis
$\BB$ for $\XX$ such that $L_m[\BB]\approx m$ for $m\in\NN$ and 
$\XX^{(2^n-2)}[\BB]=\XX^{(2^n-2)}[\BB_{0}]$ for all $n\ge 2$.
\end{lemma}
\begin{proof} The proof of Theorem 3.3 from  \cite{AAW2017} gives the result, although is not  explicitly stated there.
\end{proof}

\begin{theorem}\label{T5} Suppose $\XX$ is a Banach space with a basis of type P and let $(\SSS,\Vert \cdot \Vert_\SSS)$ be a subsymmetric Banach space. Assume that $\SSS\lesssim_c \XX$ and that $\SSS$ has nontrivial type.
Then $\XX$ has an almost greedy basis $\BB_{\kappa}$  whose fundamental function is equivalent to $(\Vert\sum_{j=1}^m \ee_j\Vert_\SSS)_{m=1}^{\infty}$
 and such that $L_m[\BB_{\kappa}]\approx \log m$ for $m\ge 2$.\end{theorem}

\begin{proof} Just combine  Lemma~\ref{AnsoLemma} with Theorem~\ref{MainDKKTheorem}.
\end{proof}

\begin{example}\
 \begin{itemize}
\item[(i)] By Proposition~\ref{PXLorentz}  and  \cite{AAW2017}*{Proposition 2.10} (which states that the unit-vector system is a basis of type P for Pisier-Xu spaces),
Theorem~\ref{T5}   applies to $\PX^0_{p,q}$ for
$1<p,q<\infty$.

\item[(ii)] The unit-vector system is a shrinking basis of type P for the  James space $\JJ^{(p)}$.
 It is also known that $\ell_p\lesssim_c \JJ^{(p)}$. Actually, the linear operator $L$  defined  in \eqref{Lifting} is bounded from $\ell_p$ into $\JJ^{(p)}$ and the linear operator $T$ defined  in \eqref{Retraction} is bounded from $\JJ^{(p)}$ into $\ell_p$ (see also \cite{BJL2013}*{Lemma 3.2}).  
So, Theorem~\ref{T5} applies both to $\JJ^{(p)}$ and  $(\JJ^{(p)})^*$.
\end{itemize}
\end{example}

\begin{theorem}\label{T31} Suppose $\XX$ is a Banach space with a basis
and that $(\SSS,\Vert \cdot \Vert_\SSS)$ be a subsymmetric Banach space. Assume that $\SSS\lesssim_c \XX$, that $\SSS$ has nontrivial type,
 and that either $\ell_1 \lesssim_c \XX$ or  $c_0 \lesssim_c \XX$.
Then $\XX$ has an almost greedy basis $\BB_{\kappa}$  whose fundamental function is equivalent to $(\Vert\sum_{j=1}^m \ee_j\Vert_\SSS)_{m=1}^{\infty}$
 and such that $L_m[\BB_{\kappa}]\approx \log m$ for $m\ge 2$.\end{theorem}

\begin{theorem}\label{T1}Let $\XX$ be a Banach space with a basis. If $\ell_1\lesssim_c \XX$ then
$\XX$ has an almost greedy basis  $\BB_{\kappa}$ with fundamental function  equivalent  to $(m)_{m=1}^\infty$
and $L_m[\BB_{\kappa}] \approx \log m$ for $m\ge 2$.
 \end{theorem}

\begin{proof}[Proof of Theorems~\ref{T31} and \ref{T1}]
Since $\ell_1\oplus \ell_1\approx \ell_1$ and $c_0\oplus c_0\approx c_0$,
we infer from Lemma~\ref{l1Lemma} (or  Lemma~\ref{Summingc0}) and Lemma~\ref{LemmaOne} that $\XX$ has a basis $\BB$ with $L_m[\BB] \approx m$ for $m\in\NN$. Appealing to Theorem~\ref{MainDKKTheorem}
(we put  $\SSS=\ell_1$ when proving Theorem~\ref{T1}) concludes the proof.
\end{proof}

\begin{example}\
\begin{itemize}
\item[(i)]  Given $1<p<\infty$, Theorem~\ref{T31} applies to the spaces $D_{p,0}$, $D_{p,1}$, $Z_{1,p}$, $Z_{p,1}$, $Z_{p,0}$, $Z_{0,p}$, and the fundamental function of the bases we obtain is $(m^{1/p})_{m=1}^\infty$.

\item[(ii)] Theorem~\ref{T31} also applies to the Hardy space $H_1$ and its predual  $\VMO$, and the almost greedy bases that we obtain have fundamental function equivalent to $(m^{1/2})_{m=1}^\infty$.

\end{itemize}
\end{example}

\begin{example}\
\begin{itemize}
\item[(i)] The list of Banach spaces for which Theorem~\ref{T1} applies includes
 $D_{1,p}$, $Z_{p,1}$ and $Z_{1,p}$ for $p\in\{0\}\cup(1,\infty)$, 
$B_{p,1}$ for $p\in(1,\infty]$, $\ell_1$, $L_1$, the Hardy space $H_1$, and the Lorentz sequence spaces $d_1(\ww)$ for $\ww$ decreasing.
\item[(ii)] By invoking \cite{LinPel1968}*{Proposition 7.3} and  \cite{JRZ1971}*{Theorem 5.1}, 
 Theorem~\ref{T1} applies to any separable $\SL_1$-space. Indeed, since $\SL_1$-spaces are GT-spaces (see \cite{LinPel1968}*{Theorem 4.1}), in light of \cite{DST2012}*{Theorem 4.2}, the conclusion on democracy is redundant for such spaces.
 \end{itemize}
 \end{example}
 
 At this moment, the only reflexive spaces $D_{p,q}$, $Z_{p,q}$ or $B_{p,q}$ not yet in our list of Banach spaces with an almost greedy basis as conditional as possible are Besov spaces $B_{1,q}$ and $B_{\infty,q}$  for $1<q<\infty$.
\begin{theorem}\label{T32} Suppose $\XX$ is a Banach space with a basis so that  either $B_{\infty,q}\lesssim_c\XX$ or $B_{1,q}\lesssim_c\XX$ for some $1<q<\infty$.  Then $\XX$ has an almost greedy basis $\BB_{\kappa}$ with
fundamental function  equivalent to $(m^{1/q})_{m=1}^\infty$  and such that $L_m[\BB_{\kappa}]\approx \log m$ for $m\ge 2$.
\end{theorem}
\begin{proof}Let $p\in\{1,\infty\}$.  Since $\ell_q\lesssim_c \XX$ and $ \ell_q \oplus B_{p,q}\approx B_{p,q}$,
$\XX$ has an almost greedy basis whose fundamental function is equivalent to
 $(m^{1/q})_{m=1}^\infty$, and we have $\XX\approx  B_{p,q}\oplus\XX$.
 Hence, taking into account Lemma~\ref{LemmaOne}, it suffices to prove the result for
$\XX=B_{p,q}$.    For $p=1$, let  $\BB$  be the difference basis in $\ell_1$, whereas for  $p=\infty$, let  $\BB$ be the summing basis of $c_0$. Let $\sigma=([2^{n-1},2^n-1])_{n=1}^\infty$. By Lemma~\ref{LemmaTwo} and Theorem~\ref{MainDKKTheorem},
the sequence $\bigoplus_{n=1}^\infty (\ee_j)_{j=1}^{2^n-1}$ is an almost greedy basis as desired for $\ZZ=(\bigoplus_{n=1}^\infty 
\YY^{(2^n-1)}[\BB,\ell_q,\sigma])_q$. By Theorem~\ref{FirstProperties}~(c), 
\[
\ZZ\approx \left(\bigoplus_{n=1}^\infty Q_\sigma(\ell_q^{2^n-1}) \oplus \ell_p^n\right)_q
\approx \VV\oplus B_{p,q},\]
where
\[\VV=\left(\bigoplus_{n=1}^\infty Q_\sigma(\ell_q^{2^n-1})\right)_q.
\]
Notice that $\VV\lesssim_c (\bigoplus_{n=1}^\infty \ell_q^{2^n-1} )_q\approx\ell_q $. Moreover, appealing to \eqref{BesovIso} gives
 $\ell_q(B_{p,q})\approx B_{p,q}$.  Hence, applying Pe{\l}czy{\'n}ski's decomposition technique yields $\ZZ\approx  B_{p,q}$.
\end{proof}

\begin{remark}By Propositions 4.1 and 4.4 from \cite{DKKT2003}, the bases obtained in Theorems~\ref{T30}, \ref{T5},
\ref{T31} and \ref{T32} are bi-democratic. Then (see  \cite{DKKT2003}*{Theorem 5.4}) their dual bases also are almost greedy.
\end{remark}

Now, in order to  complete our study, it remains to   deal with non-reflexive Besov spaces $B_{p,0}$. First, we realize that, in this case,  it is hopeless to try to obtain almost greedy bases.

\begin{theorem} Let $1\le p<\infty$.
 Then $B_{p,0}$ has no superdemocratic basis.
 In particular,  $B_{p,0}$ has no almost greedy basis.
\end{theorem}
This will follow immediately from our next Proposition, taking into account that $B_{p,0}$ and $c_0$ are not isomorphic.
\begin{proposition}\label{SteveProp}Let $(\XX_n)_{n=1}^\infty$ be a sequence a finite-dimensional Banach spaces, and $\BB$ be a superdemocratic basic sequence in $\XX=(\oplus_{n=1}^\infty \XX_n)_q$ for some  $q\in\{0\}\cup [1,\infty)$. Then
\begin{enumerate}
\item[(a)] If $q\not=0$ the fundamental function  of $\BB$ is equivalent to $(n^{1/q})_{n=1}^\infty$. 
\item[(b)] If $q=0$ then $\BB$ is equivalent to the unit vector basis of $c_{0}$.
\end{enumerate}
\end{proposition}
\begin{proof}  
Given $N\in\NN$, let $P_N\colon\XX\to\XX$ be the canonical projection onto the  first $N$ coordinates.  
With the convention $P_0=0$, put also $P_{M,N}=P_N-P_{M}$  and $P_{M,N}^c=\Id_\XX-P_{M,N}$ for $0 \le M\le N$. Let $\BB=(\xx_j)_{j=1}^\infty$ be a superdemocratic basic sequence in $\XX$.
 Since  the set $\{ P_N(\xx_j) \colon j\in\NN\}$ is compact, we have that for every $N\in\NN$ every subsequence of $\BB$ possesses a further subsequence $(\yy_j)_{j=1}^\infty$ such that $(P_N(\yy_j))_{j=1}^\infty$ converges. Fix a sequence
 $(\varepsilon_k)_{k=1}^\infty$  of positive scalars.
 Applying the gliding-hump technique,
 we infer that there are increasing sequences $(j_k)_{k=1}^\infty$ and $(N_k)_{k=1}^\infty$ of positive integers such that 
  \[
 \left\Vert P_{N_{k-1},N_k}^c (\xx_{j_{2k}}-\xx_{j_{2k-1}})\right\Vert\le \varepsilon_k,
 \]
Then, by the principle of small perturbations (see \cite{AlbiacKalton2016}*{Theorem 1.3.9}), if we choose $(\varepsilon_k)_{k=1}^\infty$ conveniently  and we put
\[
\yy_k=  P_{N_{k-1},N_k} (\xx_{j_{2k}}-\xx_{j_{2k-1}}), \quad k\in\NN,
 \]
 the basic sequences $(\yy_k)_{k=1}^\infty$  and $(\xx_{j_{2k}}-\xx_{j_{2k-1}})_{k=1}^\infty$ are equivalent. In particular, $(\yy_k)_{k=1}^\infty$ is semi-normalized and then, since it is disjointly supported,  it  is  equivalent to the canonical basis of $\ell_q$ ($c_0$ in the case when $q=0$). Therefore
\[
\left\Vert\sum_{k=1}^m \xx_{j_{2k}}-\sum_{k=1}^m \xx_{j_{2k-1}} \right\Vert \approx 
\left\Vert\sum_{k=1}^m \yy_k \right\Vert \approx m^{1/q}, \quad m\in\NN
\]
(with the usual modification when $q=0$). Since $\phi_m[\BB]\le \phi_{2m}[\BB] \le 2 \phi_m[\BB]$, 
the superdemocracy of $\BB$ yields $\phi_m[\BB]\approx m^{1/q}$ for $m\in\NN$ in the case when $q\ge 1$ and
$\phi_m[\BB]\approx 1$ for $m\in\NN$ in the case when $q=0$. 
The proof of (a) is over, and from here (b) is straightforward.
\end{proof}

\begin{theorem}\label{T8}Let $\XX$ be a Banach space with a quasi-greedy basis.  Assume that $B_{p,0}\lesssim_c \XX$ for some $1\le p<\infty$, then $\XX$ has a quasi-greedy basis $\BB_{\kappa}$ with $L_m[\BB_{\kappa}]\approx \log m$ for $m\ge 2$.
\end{theorem}

\begin{proof}The relations \eqref{PeuRad} and \eqref{BesovIso} give $B_{2,0}\lesssim_c B_{p,0}$ for $1<p<\infty$. Moreover, since $B_{p,0}\approx B_{p,0}\oplus B_{p,0}$ we have $B_{p,0}\oplus\XX\approx \XX$. Then,  by Lemma~\ref{LemmaOne}, it suffices show to the existence of a quasi-greedy basis as conditional as possible  in the space $B_{p,0}$  in the case when $p\in\{1,2\}$.
Let   $\sigma=([2^{n-1},2^n-1])_{n=1}^\infty$ and $\SSB$ be the summing basis of $c_0$. 
Combining   Lemma~\ref{Summingc0}, Lemma~\ref{LemmaTwo} and Theorem~\ref{MainDKKTheorem} we obtain that 
$\BB_{\kappa}=(\bigoplus_{n=1}^\infty   (\ee_j)_{j=1}^{2^n-1})_0$ is a quasi-greedy basis for the space
\[
\ZZ:=\left(\bigoplus_{n=1}^\infty  \YY^{(2^n-1)}[\SSB,\ell_p,\sigma]\right)_0
\] such that $L_m[\BB_{\kappa}]\approx \log m$
for $m\ge 2$.  
By Theorem~\ref{FirstProperties}~(c), if we put 
\[
\VV=\left(\bigoplus_{n=1}^\infty Q_\sigma(\ell_p^{2^n-1})\right)_0,
\]
we have
\[
\ZZ\approx \VV\oplus (\oplus_{n=1}^\infty \ell_\infty^n)_0\approx \VV\oplus c_0\approx\VV.
\]
In the case when $p=2$ we have $Q_\sigma(\ell_q^{2^n-1}) \approx \ell_2^{2^n-n-1}$ for $n\in\NN$. 
 Assume that $p=1$ and let $B=\NN\setminus\{2^n-1\colon n\in\NN\}$.   The coordinate projection $S_B$ restricts
 to an isomorphism from $Q_\sigma(\ell_1)$ onto
\[
\WW=\left\{ a_n)_{n=1}^\infty \in\ell_1 \colon a_{2^n-1}=0 \text{ for all } n\in \NN\right\}.
\]
Hence, $Q_\sigma(\ell_1^{2^n-1})\approx \ell_1^{2^n-n-1}$ for $n\in\NN$. In both cases,
 taking into account \eqref{BesovIso}, we have $\VV\approx (\bigoplus_{n=1}^\infty \ell_p^{2^n-n-1}))_0\approx B_{p,0}$.
\end{proof}



\end{document}